\documentclass[12pt]{article}
\usepackage{amsmath, amssymb, amsthm}
\usepackage{mathtools}    
\usepackage{mathdots}      
\usepackage{mathrsfs}

\usepackage{geometry}

\usepackage{xcolor}        
\usepackage{graphicx,url,hhline}
\usepackage[english]{babel}

\usepackage{longtable,array,multirow}
\usepackage{booktabs}

\usepackage{tikz}
\usetikzlibrary{arrows.meta, shapes.geometric, positioning, calc}

\usepackage{enumitem}
\usepackage{indentfirst}
\usepackage{adjustbox}

\usepackage{algorithm}
\usepackage{algorithmicx}
\usepackage{algpseudocode}

\usepackage{subcaption}
\usepackage{caption}
\usepackage{float, placeins}
\usepackage{etoolbox}

\usepackage{hyperref} 
\makeatletter
\let\org@cite\cite
\renewcommand{\cite}[1]{\textbf{\org@cite{#1}}}
\makeatother
\newtheorem{theorem}{Theorem}[section]
\newtheorem{lemma}[theorem]{Lemma}

\newtheorem{definition}[theorem]{Definition}
\newtheorem{remark}[theorem]{Remark}
\newtheorem{example}[theorem]{Example}

\counterwithin{figure}{section}  
\numberwithin{equation}{section}

\newtheorem*{Main Theorem}{Main Theorem}

\tikzset{
    arrow/.style={->, >=Stealth, thick},    
    node/.style={circle, draw=none, fill=none, font=\tiny}, 
    red-bar/.style={line width=3pt, red},   
    mut-label/.style={midway, font=\scriptsize, #1}, 
    blue/.style={color=blue},
    green/.style={color=green!70!black},
    purple/.style={color=purple},
    black/.style={color=black},
}

\title{\large \textbf{Effective Computation of Mutation Paths and Generators of Cluster Automorphism Groups}}
\author{Jindong Zhao and Haiyan Zhu  \\ School of Mathematical Sciences, Zhejiang University of Technology}

\begin{document}
\date{}
\maketitle

\begin{abstract}
In this paper, we improve the marked-vertex strategy introduced by Fu and Liang,  and design the algorithm to compute all mutation paths and generator elements of cluster automorphism groups efficiently. As an application, we get generators of cluster automorphism groups of cluster algebras of finite mutation type of rank 4.
\end{abstract}


\renewcommand{\thefootnote}{} 

\footnote{Supported by the National Natural Science Foundation of China (12271481).}

\tableofcontents  

\section{Introduction}
Cluster algebras, introduced by Fomin and Zelevinsky~\cite{I. Assem}, are commutative rings endowed with a combinatorial structure determined by seeds and mutations.
Their rich combinatorial and algebraic properties establish profound connections to a wide range of mathematical fields, including representation theory, integrable systems, and algebraic geometry.
In~\cite{I. Assem}, Assem, Schiffler and Shramchenko introduced and characterized cluster automorphisms of a cluster algebra, and computed the cluster automorphism groups for finite and affine types.
Since then, cluster automorphism groups have been extensively investigated.
Blanc and Dolgachev~\cite{bd15} used geometric methods to determine the cluster automorphism groups of cluster algebras of rank two.
Chang and Zhu~\cite{chang2016cluster-finite} explicitly described the cluster automorphism groups of finite type cluster algebras via folding techniques.
The cluster automorphism groups for cluster algebras arising from marked surfaces have been further studied in~\cite{Gu11,BS15,DL25}.
Recently, Fu and Liang~\cite{FuLiang} used a pseudo $\mathbb{N}$-grading to compute the generating sets of cluster automorphism groups for cluster algebras of rank~$3$.

However, the computation of $\operatorname{Aut}(\mathcal{A})$ of cluster algebras of rank 4 presents substantial challenges:
(1) the lack of a complete classification framework,
(2) severe combinatorial explosion and intrinsic computational complexity in determining generating sets,
(3) the highly intricate polynomial relations among potential generators render manual verification and simplification a significant challenge, even for small rank 4 cases.

Motivated by the approach of Fu and Liang, we proceed with a similar strategy.  
For the cluster algebra of Dynkin type $\mathbb{D}_4$, the mutation graph contains multiple disjoint cyclic paths, as illustrated in Figure~\ref{fig:4-node-added}.  
Consequently, the procedure becomes trapped in an infinite loop under Assumptions ($\diamond$).  

To resolve this issue and establish a valid termination condition, we improve the \textbf{marked-vertex strategy} to refine the generating set constructed by Fu and Liang.  
We then design effective algorithms to compute the generating sets of cluster automorphism groups. And as an application, we get generators  of finite mutation type cluster algebras of rank 4.  
Subsequently, we eliminate redundant generators — a process similar to computing a reduced Gr\"{o}bner basis, which reveals the importance of essential generators — to obtain the essential generators of each cluster automorphism group.

The paper is organized as follows. In Section~\ref{s2}, we recall the basic definitions relative to cluster algebras and cluster automorphisms. 
Then we use an example to illustrate Fu-Liang's method and highlight the necessity of algorithm optimization in cluster algebras of rank 4. In Section~\ref{s4}, we generalize the approach of Fu-Liang by adding marked vertices in Theorems~\ref{thm:refined-generator00}. After introducing the marked-vertex strategy, we establish algorithms for calculating cluster automorphism groups. In Appendix~\ref{s3}, according to \cite{felikson2012cluster}, all diagrams of finite-mutation skew symmetrizable matrices are listed in Table~\ref{jjjsblocks}.
As an application of Algorithms~\ref{iii}, we determine generators of the cluster automorphism groups of finite mutation types of rank 4 with indecomposable exchange matrices in Section~\ref{s5}. Furthermore, we characterize the relations among generators.

\newpage

\section{Preliminaries}
\label{s2}

In this section, we recall basic definitions in cluster algebras and restrict our attention to cluster algebras with trivial coefficients, which were introduced by Fomin and Zelevinsky in \cite{FominZelevinsky2002} and let \( n \) be a positive integer.

\subsection{Cluster algebras}

\begin{definition}
For an \( n \times n \) integer matrix \( B = (b_{ij})_{n \times n} \):  
\begin{enumerate}[label=(\arabic*)]
    \item \( B \) is \textbf{skew-symmetric} if \( B^\top = -B \).  
    \item \( B \) is \textbf{skew-symmetrizable} if there exists a positive-integer diagonal matrix \( D = \operatorname{diag}(d_1, d_2, \dots, d_n) \) such that \( DB = (d_i b_{ij})_{n \times n} \) is skew-symmetric. In this case, \( D \) is called a \textbf{skew-symmetrizer} of \( B \).  
\end{enumerate}
\end{definition}

\begin{definition}
Let \(\mathcal{F}\) be a purely transcendental extension field of \(\mathbb{Q}\). A \textbf{labeled} seed is a pair \((\mathbf{x}, B)\) consisting of:
\begin{itemize}
    \item \(\mathbf{x} = (x_1, \ldots, x_n)\), a transcendence basis of \(\mathcal{F}\) over \(\mathbb{Q}\); elements of \(\mathbf{x}\) are called \textbf{cluster variables}, and \(\mathbf{x}\) itself is termed a \textbf{cluster}.
    \item \(B = (b_{ij})\), an \(n \times n\) skew-symmetric integer matrix, referred to as the \textbf{exchange matrix}. 
\end{itemize}
\end{definition}

\begin{definition}
Given a seed \((\mathbf{x}, B)\) and an index \(k \in \{1, \cdots, n\}\), the mutation \(\mu_k\) in direction \(k\) transforms \((\mathbf{x}, B)\) into a new seed \((\mathbf{x}', B') = \mu_k(\mathbf{x}, B)\), where:\\
- \textbf{Mutation of cluster variables}: The new cluster \(\mathbf{x}' = (x_1', \cdots, x_n')\) satisfies \(x_i' = x_i\) for all \(i \neq k\), while \(x_k'\) is determined by the exchange relation:
\[
x_k x_k' = \prod_{\substack{i \\ b_{ik} > 0}} x_i^{b_{ik}} + \prod_{\substack{i \\ b_{ik} < 0}} x_i^{-b_{ik}}.
\]
- \textbf{Mutation of the exchange matrix}: The entries of the new exchange matrix \(B' = (b_{ij}')\) are given by:
\[
b_{ij}' = \begin{cases} 
-b_{ij} & \text{if } i = k \text{ or } j = k, \\
b_{ij} + \frac{1}{2} \left( |b_{ik}| b_{kj} + b_{ik} |b_{kj}| \right) & \text{otherwise}.
\end{cases}
\]
\end{definition}

The mutation is an involution, i.e.,  $\mu_k\mu_k(\mathbf{x}, B) =(\mathbf{x}, B)$.
Let \( \mathbb{T}_n \) be the \( n \)-regular tree whose edges are labeled by the numbers \( 1, \ldots, n \) such that the \( n \) edges emanating from each vertex carry distinct labels. For each vertex \( t \in \mathbb{T}_n \), there is a unique reduced sequence \( i_k \ldots i_1 \) such that \( t_0 \) and \( t \) are connected by edges labeled by \( i_1, \ldots, i_k \), i.e., $
t_0 \stackrel{i_1}{\rule[0.5ex]{1em}{0.4pt}} \cdot \stackrel{i_2}{\rule[0.5ex]{1em}{0.4pt}} \cdots \stackrel{i_k}{\rule[0.5ex]{1em}{0.4pt}} t.
$

\begin{definition}
The \textbf{cluster algebra} \(\mathcal{A} = \mathcal{A}(\mathbf{x}_{t_0}, B_{t_0})\) associated with an initial seed \((\mathbf{x}_{t_0}, B_{t_0})\) is the subalgebra of \(\mathcal{F}\) generated by all cluster variables from seeds obtained via arbitrary sequences of mutations starting from \((\mathbf{x}_{t_0}, B_{t_0})\).

Moreover, a cluster algebra is called of finite type if it has finitely many cluster variables, while a cluster algebra with only finite number of distinct exchange matrices is called of the finite mutation type.
\end{definition}

\begin{definition}
\label{def:cluster-pattern}
A collection of labeled seeds $\boldsymbol{\Sigma} = \{\Sigma_t = (\mathbf{x}_t, B_t)\}_{t \in \mathbb{T}_n}$ indexed by the $n$-regular tree $\mathbb{T}_n$ is called a \textbf{cluster pattern} or \textit{seed pattern} of rank $n$ if for any pair of vertices $t, t' \in \mathbb{T}_n$ that are $k$-adjacent, i.e., $\Sigma_{t'} = \mu_k(\Sigma_t)$.
\end{definition}

\subsection{Diagrams of exchange matrices}
In this section, we recall the diagrams of exchange matrices, which were introduced by Fomin and Zelevinsky in \cite{FZ03}.

\label{subsec:diagrams}
\begin{definition}
\label{def:diagram}
Let \( B = (b_{ij})_{i,j=1}^n \) be an \( n \times n \) skew-symmetrizable matrix. The \textit{diagram} \( \Gamma(B) \) associated with \( B \) is a weighted directed graph where the vertex set is \( \{1, \ldots, n\} \). A directed edge from vertex \( i \) to vertex \( j \) exists precisely when \( b_{ij} > 0 \), and such an edge is assigned a weight of \( |b_{ij}b_{ji}| \). We call \( B \) \textit{indecomposable} if its corresponding diagram \( \Gamma(B) \) is connected.
\end{definition}

\begin{definition}
\label{def:diagram-mutation}
For any vertex \( k \) in a diagram \( \Gamma := \Gamma(B) \) (where \( B \) is a skew-symmetrizable matrix), the mutation \( \mu_k(\Gamma) \) of \( \Gamma \) at \( k \) is a weighted directed graph constructed from \( \Gamma \) through the following steps:
\begin{itemize}
    \item \textbf{Edge reversal:} All edges incident to vertex \( k \) have their orientations reversed, with weights preserved.
    \item \textbf{Path composition:} For any pair of vertices \( i \) and \( j \) connected via a 2-edge path $i \to k \to j$ or $i \leftarrow k \leftarrow j$ in $\Gamma$:
    \begin{align*}
        \pm \sqrt{a} \pm \sqrt{a'} = \sqrt{bc}
    \end{align*}
    
    \begin{itemize}[label={}]
        \item $b$ = weight of edge $(i,k)$ or $(k,i)$, $c$ = weight of edge $(k,j)$ or $(j,k)$
        \item $a$ = weight of edge $(i,j)$ , $a'$ = new weight of edge $(i,j)$ 
        \item Signs determined by cycle orientations (The sign before \( a \) is ``+'' if $i \to k \to j$, and is ``-'' otherwise. And the sign before \( a' \) is opposite to that before \( a \) ) 
    \end{itemize}
    \item \textbf{Preservation:} All other edges and weights remain unchanged.
\end{itemize}
\textit{This diagram mutation rule encodes the signed adjacency changes during matrix mutation.}
\end{definition}

\begin{definition}
\label{def:sn-action}
Let \( \Sigma = (\mathbf{x}, B) \) be a labeled seed of rank \( n \), and let \( \Gamma(B) \) be the diagram corresponding to \( B \). For any \( \sigma \in S_n \):
\begin{itemize}
    \item The action on the seed is:
    \[
    \sigma (\Sigma) = (\sigma \mathbf{x}, \sigma (B)), \quad 
    \begin{cases} 
    x_i' = x_{\sigma^{-1}(i)} \\ 
    b_{ij}' = b_{\sigma^{-1}(i) \sigma^{-1}(j)}
    \end{cases}
    \]
    \item The action on the diagram $ \sigma(\Gamma(B)) $ relabels vertices via $\sigma$ while preserving edge weights and directions between corresponding vertices.
\end{itemize}
\textit{This group action formalizes vertex relabeling and is compatible with mutation operations.}
\end{definition}

Furthermore, for skew-symmetrizable matrices \( B \) and \( B' \), we write \( B \sim B' \) if there exists a \( \sigma \in S_n \) such that \( B' = \sigma(B) \).
In the following, $\Gamma(B)$ and  $\Gamma(B')$ are said to be \textbf{sign permutation equivalent} if $B' \sim \pm B $.

\begin{lemma}
\label{lem:sn-action-properties}
Let \( \Sigma = (\mathbf{x}, B) \) be a labeled seed of rank \( n \), let \( \Gamma(B) \) be the diagram of \( B \), and let \( 1 \leq k \leq n \). The following statements hold for any \( \sigma \in S_n \):
\begin{enumerate}
    \item \( \sigma(\Gamma(B)) = \Gamma(\sigma(B)) \) (diagram isomorphism)
    \item The actions of mutation and permutation commute: 
    \(
    \sigma \circ \mu_k = \mu_{\sigma(k)} \circ \sigma.
    \)
    That is:
    \[
    \sigma(\mu_k(\Sigma)) = \mu_{\sigma(k)}(\sigma(\Sigma)) \quad \text{and} \quad \sigma(\mu_k(\Gamma(B))) = \mu_{\sigma(k)}(\sigma(\Gamma(B))).
    \]
\end{enumerate}
\end{lemma}

\subsection{Cluster automorphism}
In this section, we recall the approach introduced by Fu and Liang \cite{FuLiang} for calculating cluster automorphism groups of cluster algebras.

Firstly, we recall the definition of cluster automorphisms, which were introduced by Assem, Schiffler and Shamchenko in \cite{I. Assem}.

\begin{definition}[\cite{I. Assem}]
\label{def:cluster-automorphism}
Let $\mathcal{A}=\mathcal{A}(\Sigma) :=\mathcal{A} (\mathbf{x}, B)$ be a cluster algebra, and $f: \mathcal{A} \to \mathcal{A}$ be an automorphism of $\mathbb{Z}$-algebra. $f$ is called a \textbf{cluster automorphism} of $\mathcal{A}$ if there exist another seed $\Sigma' := (\mathbf{x}', B')$ of $\mathcal{A}$ and a permutation $\sigma \in S_n$, such that:
\begin{enumerate}[label=(\arabic*)]
    \item $\sigma f(\mathbf{x}) = \mathbf{x}',$
    \item $\sigma(f(\mu_k(\mathbf{x}))) = \mu_{\sigma(k)}(\mathbf{x}')$ for $k = 1, \ldots, n.$
\end{enumerate}
\end{definition}

In fact, the condition $(2)$ is superfluous that was conjecture by Chang and Schiffler and proved in~\cite{li}.

In fact, every cluster automorphism of \( \mathcal{A} \) can be written as a quadruple  
\[
f = (\Sigma_1, \Sigma_2, \sigma, \varepsilon),
\]  
where \( \Sigma_1 = (\mathbf{x}_1, B_1) \) and \( \Sigma_2 = (\mathbf{x}_2, B_2) \) are labeled seeds, \( \sigma \in S_n \), and \( \varepsilon \in \{ \pm \} \) such that \( \sigma B_1 = \varepsilon B_2 \). The automorphism \( f \) is uniquely determined by the condition  
$\sigma f(\mathbf{x}_1) = \mathbf{x}_2$,
which was proven in \cite{FuLiang}.

Denote the set of all cluster automorphisms of $\mathcal{A}$ by $\operatorname{Aut}(\mathcal{A})$ which forms a group under composition, cf.\ \cite{I. Assem}.

Recall that, we have, fixed a cluster pattern $\boldsymbol{\Sigma} = \{\Sigma_t = (\mathbf{x}_t, B_t)\}_{t \in \mathbb{T}_n}$ with root vertex $t_0$, for any two vertices $t, t' \in \mathbb{T}_n$, let $p(t, t')$ be the unique path in $\mathbb{T}_n$ from $t$ to $t'$. If  \( p(t, t') = t \stackrel{i_1}{\rule[0.5ex]{1em}{0.4pt}} \cdot \stackrel{i_2}{\rule[0.5ex]{1em}{0.4pt}} \cdot \cdots \cdot \stackrel{i_m}{\rule[0.5ex]{1em}{0.4pt}} t'\), for any \( \sigma \in S_n \), we denote \[ \mu_{p(t,t')} := \mu_{i_m} \cdots \mu_{i_1}  \text{ and }  \mu_{\sigma(p(t,t'))} := \mu_{\sigma(i_m)} \cdots \mu_{\sigma(i_1)}. \] 

For a fixed vertex $s \in \mathbb{T}_n$, we define the weight of $p(t, t')$ with respect to $s$ as follows
\[
\mathbf{w}_s(t, t') := \# \left\{ t'' \in \mathbb{T}_n \mid \text{the path } p(t, t') \text{ passes through the vertex } t'' \text{ and } B_{t''} \sim \pm B_s \right\}.
\]

For any positive integer $m$, we define
\[
\mathcal{P}_m (t_0) := \{ p(t_0, t) \mid \mathbf{w}_{t_0} (t_0, t) = m + 1 \text{ and } B_t \sim \pm B_{t_0} \}.
\]

For each path $p(t_0, t) \in \mathcal{P}_m (t_0)$, we fix a permutation $\sigma_t \in S_n$ and $\varepsilon_t \in \{\pm\}$ such that $\sigma_t (B_{t_0}) = \varepsilon_t B_t$, and denote by $f_{p(t_0, t)} := (\Sigma_{t_0}, \Sigma_t, \sigma_t, \varepsilon_t)$ the associated cluster automorphism.

Furthermore, for a certain cluster automorphism $f=(\Sigma_{t_0}, \Sigma_t, \sigma, \varepsilon)$, by definition, we know that there is a unique reduced sequence \( i_k \cdots i_1 \) such that \( t = \mu_{i_k} \cdots \mu_{i_1}(t_0) \). So, we denote the cluster automorphism as  
\[
g_{i_k \cdots i_1}^{\sigma, \varepsilon} \coloneqq (\Sigma_{t_0}, \Sigma_t, \sigma, \varepsilon)
\text{ or }
\psi_\sigma^\varepsilon \coloneqq (\Sigma_{t_0}, \Sigma_{t_0}, \sigma, \varepsilon).
\]

Let $m$ be a positive integer, define
\[
H_m (t_0) := \{ f_{p(t_0, t)} \mid p(t_0, t) \in \mathcal{P}_m (t_0) \} \subset G_m (t_0),
\]
where
\[
G_m (t_0) := \left\{ f \in \operatorname{Aut}(\mathcal{A}) \middle| \begin{array}{l} \exists s \in \mathbb{T}_n, \sigma \in S_n \text{ and } \varepsilon \in \{\pm\} \\ \text{such that } \mathbf{w}_{t_0} (t_0, s) = m + 1 \text{ and } \\ f = (\Sigma_{t_0}, \Sigma_s, \sigma, \varepsilon) \end{array} \right\} .
\]

Assume that there is a (marked) vertex $t_\diamond$ satisfyies 

\begin{center}
\textbf{Assumption ($\diamond$)}: there exists a vertex $t_\diamond \in \mathbb{T}_n$ such that $B_{t_\diamond} \not\sim \pm B_{t_0}$ and $\mathbf{w}_{t_0} (t_0, t_\diamond) = \mathbf{w}_{t_\diamond} (t_0, t_\diamond) = 1$.
\end{center}

For any $n,m \geq 0$, we define
\begin{align*}
\mathcal{P}_{n,m}^{t_\diamond} (t_0) &:= \{ p(t_0, t) \in \mathcal{P}_n (t_0) \mid \mathbf{w}_{t_\diamond} (t_0, t) = m \} \subset \mathcal{P}_n (t_0), \\
H_{n,m}^{t_\diamond} (t_0) &:= \{ f_{p(t_0,t)} \mid p(t_0, t) \in \mathcal{P}_{n,m}^{t_\diamond} (t_0) \} \subset H_n (t_0).
\end{align*}

\begin{theorem}[\cite{FuLiang}]
\label{thm:refined-generator2}
The cluster automorphism group $\operatorname{Aut}(\mathcal{A})$ is generated by $G_0 (t_0) \cup H_1 (t_0)$. Forthermore, $\operatorname{Aut}(\mathcal{A})$ is generated by $G_0 (t_0) \cup H_{1,0}^{t_\diamond} (t_0) \cup H_{1,1}^{t_\diamond} (t_0) \cup H_{1,2}^{t_\diamond} (t_0)$.
\end{theorem}

\begin{example}
\label{sec:example}
Let $\mathcal{A}$ be the cluster algebra of type $\mathbb{D}_4$ and the associated exchange matrix at vertex $t_0$ is

\[
B_{t_0} = 
\begin{bmatrix}
0 & 1 & 0 & 0 \\
-1 & 0 & -1 & -1 \\
0 & 1 & 0 & 0 \\
0 & 1 & 0 & 0
\end{bmatrix}.
\]

By mutating in directions 1 and 2, we obtain the mutation subgraph of \(\Gamma_0 (=\Gamma(B_{\{t_0\}}))\), cf. Figure~\ref{fig:4-node-added} where it is easy to complete the mutation subgraphs in other directions by symmetry.

Clearly, \( t_0 \stackrel{}{\rule[0.5ex]{1em}{0.4pt}} \mu_2(t_0) \in \mathcal{P}_1(t_0)  \), then $f_{p(t_0,\mu_2(t_o))} \in H_1(t_0) $.
From Figure~\ref{fig:4-node-added}, all green, red and black vertices satisfy the Assumptions($\diamond$). If we let \( t_{\diamond} = \mu_1(t_0) \),
then  \[ t_0 \stackrel{1}{\rule[0.5ex]{1em}{0.4pt}} \cdot \stackrel{2}{\rule[0.5ex]{1em}{0.4pt}}\cdot \stackrel{3}{\rule[0.5ex]{1em}{0.4pt}} \cdot \stackrel{4}{\rule[0.5ex]{1em}{0.4pt}}  
\cdot \stackrel{2}{\rule[0.5ex]{1em}{0.4pt}} \cdot \stackrel{1}{\rule[0.5ex]{1em}{0.4pt}} t'\ (i.e.\ t'=\mu_{124321(t_0)}) \] 
belongs to \( P_{1,2}^{t_{\diamond}}(t_0) \) and \( f_{P(t_0, t')} \in H_{1,2}^{t_\diamond}(t_0)\). Note that there are three green vertices in the path, then cluster automorphisms determined by paths through \( t_0 \stackrel{1}{\rule[0.5ex]{1em}{0.4pt}} \cdot \stackrel{3}{\rule[0.5ex]{1em}{0.4pt}} \cdot \stackrel{1}{\rule[0.5ex]{1em}{0.4pt}}\) can be deleted from the generator set.

To calculate the cluster automorphisms is very difficult and the program may get stuck in an infinite loop for the path through  \[ t_0 \stackrel{1}{\rule[0.5ex]{1em}{0.4pt}} \cdot \stackrel{2}{\rule[0.5ex]{1em}{0.4pt}}\cdot \stackrel{3}{\rule[0.5ex]{1em}{0.4pt}} \cdot \stackrel{4}{\rule[0.5ex]{1em}{0.4pt}} \cdot \stackrel{3}{\rule[0.5ex]{1em}{0.4pt}} \cdot \stackrel{4}{\rule[0.5ex]{1em}{0.4pt}} \cdots, \] when we mark a green vertex.

If we mark a red or black vertex, then the similar problems will occur as above case.
No matter which marker we choose, the process cannot terminate within a finite number of steps for $\mathbb{D}_4$, and so we need to improve the algorithms.

\begin{figure}[H]
\centering
\adjustbox{max width=\textwidth}{
\begin{tikzpicture}[>=Stealth, line width=0.8pt, font=\small]

\def\drawNode#1#2#3#4#5#6#7#8{ 
  \node[font=\bfseries] at (#1, #2+1.5) {4};   
  \node[font=\bfseries] at (#1, #2)     {2};   
  \node[font=\bfseries] at (#1-1.5, #2-1.5) {1}; 
  \node[font=\bfseries] at (#1+1.5, #2-1.5) {3}; 
  
  \ifnum#3=1  \draw[->, #7] (#1, #2+1.2) -- (#1, #2+0.3) node[midway, left] {};\fi
  \ifnum#3=-1 \draw[->, #7] (#1, #2+0.3) -- (#1, #2+1.2) node[midway, left] {};\fi
  
  \ifnum#4=1  \draw[->, #7] (#1-0.3, #2-0.3) -- (#1-1.2, #2-1.2) node[midway, below left] {};\fi
  \ifnum#4=-1 \draw[->, #7] (#1-1.2, #2-1.2) -- (#1-0.3, #2-0.3) node[midway, above left] {};\fi
  
  \ifnum#5=1  \draw[->, #7] (#1+0.3, #2-0.3) -- (#1+1.2, #2-1.2) node[midway, below right] {};\fi
  \ifnum#5=-1 \draw[->, #7] (#1+1.2, #2-1.2) -- (#1+0.3, #2-0.3) node[midway, above right] {};\fi
  
  \ifnum#6=1  \draw[->, #7] (#1-1.2, #2-1.5) -- (#1+1.2, #2-1.5) node[midway, below] {};\fi
  \ifnum#6=-1 \draw[->, #7] (#1+1.2, #2-1.5) -- (#1-1.2, #2-1.5) node[midway, below] {};\fi
  
  \ifnum#8=1  
    \draw[->, #7] (#1-0.2, #2+1.2) -- (#1-1.4, #2-1.2) node[midway, above left] {};
  \fi
  \ifnum#8=-1 
    \draw[->, #7] (#1-1.4, #2-1.2) -- (#1-0.2, #2+1.2) node[midway, below right] {};
  \fi
}

\node[font=\bfseries\large] at (-2, 4.5) {$t_0:$};

\drawNode{0}{4}{1}{-1}{-1}{0}{blue}{0}
\draw (1.5,4.5) -- (3.5,4.5) node[midway,above] {$\mu_1$};
\drawNode{5}{4}{1}{1}{-1}{0}{green}{0}
\draw (6.5,4.5) -- (8.5,4.5) node[midway,above] {$\mu_3$};
\drawNode{10}{4}{1}{1}{1}{0}{green}{0}
\draw (11.5,4.5) -- (13.5,4.5) node[midway,above] {$\mu_1$};
\drawNode{15}{4}{1}{-1}{1}{0}{green}{0}

\drawNode{0}{-1}{-1}{1}{1}{0}{blue}{0}
\draw (0,1.2) -- (0,2.4) node[midway,right] {$\mu_2$};
\drawNode{5}{-1}{-1}{0}{-1}{1}{black}{1}
\draw (6.5,-1) -- (8.5,-1) node[midway,above] {$\mu_3$};
\draw (6.5,1.7) -- (8.5,0.2) node[midway,right] {$\mu_2$};
\drawNode{10}{-1}{-1}{-1}{1}{-1}{red}{1}
\drawNode{15}{-1}{-1}{-1}{1}{0}{green}{0}
\draw (15,1.2) -- (15,2.4) node[midway,right] {$\mu_4$};
\draw (15,-2.8) -- (15,-4) node[midway,left] {$\mu_3$};

\drawNode{0}{-6}{1}{1}{-1}{1}{red}{-1}
\draw (1.5,-4.8) -- (3.5,-3.3) node[midway,left] {$\mu_4$};
\draw (1.5,-6) -- (3.5,-6) node[midway,above] {$\mu_3$};
\drawNode{5}{-6}{1}{0}{1}{-1}{black}{-1}
\draw (6.5,-6) -- (8.5,-6) node[midway,above] {$\mu_4$};
\drawNode{10}{-6}{-1}{-1}{1}{-1}{red}{1}
\drawNode{15}{-6}{-1}{-1}{-1}{0}{green}{0}

\drawNode{0}{-11}{-1}{-1}{1}{0}{green}{0}
\draw (0,-9) -- (0,-8) node[midway,right] {$\mu_2$};

\drawNode{5}{-11}{-1}{1}{1}{0}{blue}{0}
\draw  (1.5,-11) -- (3.5,-11)  node[midway,above] {$\mu_3$};

\end{tikzpicture}
}
\caption{Mutation subgraph of $\Gamma_0$ of $\mathbb{D}_4.$ (All the diagrams with same color is sign permutation equivalent.)}
\label{fig:4-node-added}
\end{figure}
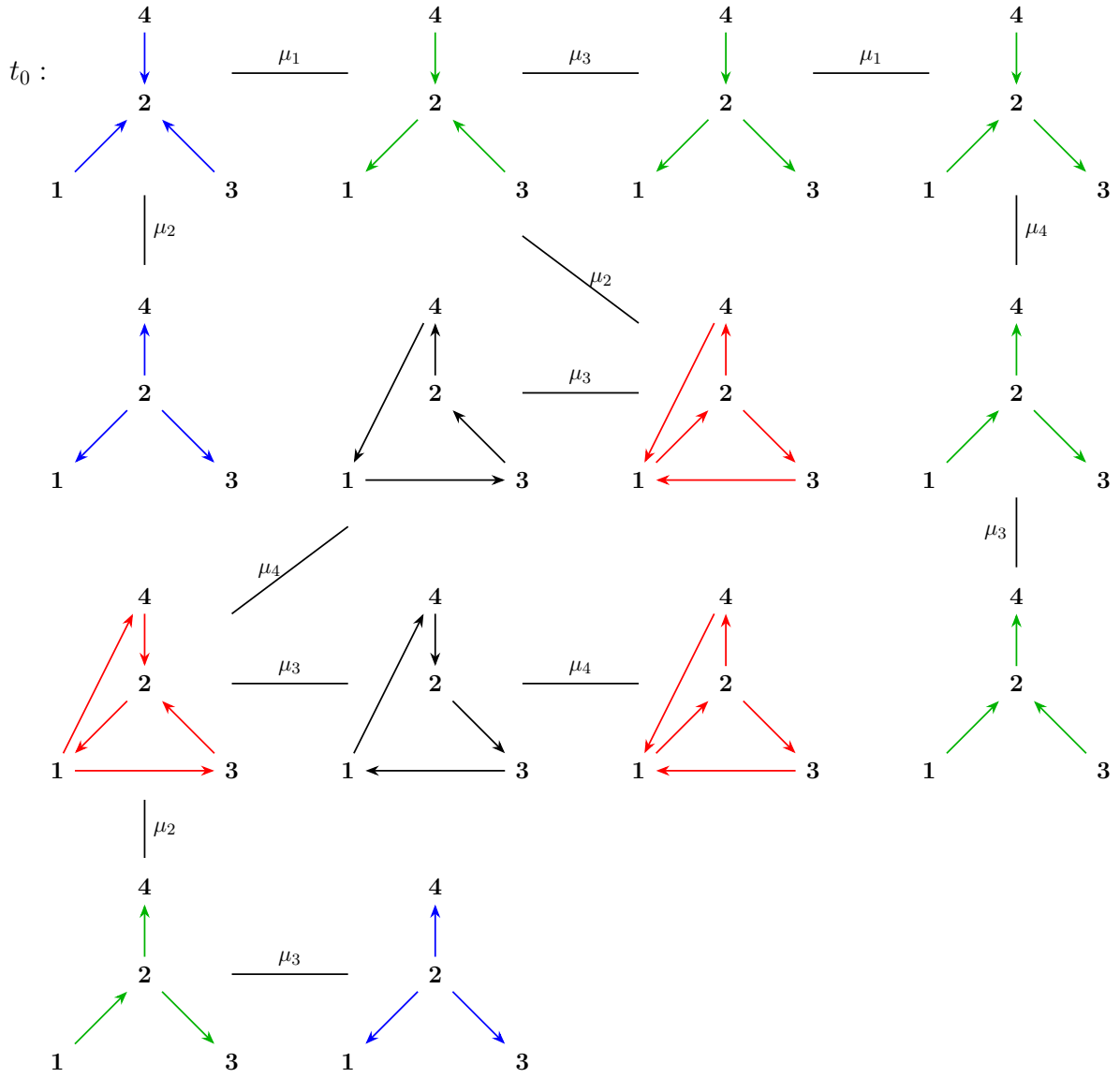

\end{example}

\newpage

\section{Marked-vertex strategy and algorithms}
\label{s4}
In this section, let \( \mathcal{A} =\mathcal{A} (\mathbf{x}, B)\) be a cluster algebra of rank \( n \) with a fixed cluster pattern \( \Sigma \). Assume that the exchange matrices of \( \mathcal{A} \) are indecomposable and fix a skew-symmetrizer \( D = \operatorname{diag}\{d_1, \ldots, d_n\} \).

\subsection{Multiple marked vertices strategy}

Inspired by the Assumption($\diamond$), we denote $t_0 \triangleq t_{\diamond_0}$, $t_{\diamond} \triangleq t_{\diamond_1} (n >2)$ and fix multiple (marked) vertices $t_{\diamond_1}, t_{\diamond_2} \cdots t_{\diamond_n}$ satisfying assumption ($\diamond_n $) as follows

\begin{center}
\textbf{Assumption ($\diamond_n $)}: there exist vertices $t_{\diamond_1}, t_{\diamond_2} \cdots t_{\diamond_n}  \in \mathbb{T}_n$ such that $B_{t_{\diamond_i}}  \not\sim \pm B_{t_{\diamond_j}} $ and $\mathbf{w}_{t_{\diamond_i}} (t_0, t_{\diamond_i}) =\mathbf{w}_{t_0} (t_0, t_{\diamond_i}) = 1$ for any $0 \leq i,j \leq n $,  and 
if $i < j$: $\mathbf{w}_{t_{\diamond_i}} (t_0, t_{\diamond_j}) = 0.$
\end{center}

\begin{lemma}[\cite{FuLiang}]\label{rem:3.10}
Let \( f = (\Sigma_1, \Sigma_2, \sigma, \varepsilon) \) be a cluster automorphism of \( \mathcal{A} \).
\begin{enumerate}[label=(\arabic*)] 
    \item For any mutation sequence \( \mu_{i_k}, \ldots, \mu_{i_1} \), the following equality holds:
    \[
    (\Sigma_1, \Sigma_2, \sigma, \varepsilon) = \bigl( \mu_{i_k} \cdots \mu_{i_1} (\Sigma_1), \mu_{\sigma(i_k)} \cdots \mu_{\sigma(i_1)} (\Sigma_2), \sigma, \varepsilon \bigr).
    \]
    \item Let \( g = (\Sigma_2, \Sigma_3, \tau, \varepsilon') \) be another cluster automorphism of \( \mathcal{A} \). Then:
    \[
    (\Sigma_2, \Sigma_3, \tau, \varepsilon') \circ (\Sigma_1, \Sigma_2, \sigma, \varepsilon) = (\Sigma_1, \Sigma_3, \tau\sigma, \varepsilon'\varepsilon).
    \]
    \item If f is a cluster automorphism determined by the path \( p(t_1, t_2) \), we have the following factorization:
    \[
    f = f_{p(t_1, t_2)} \circ (\Sigma_{t_0}, \Sigma_{t_0}, \sigma_t^{-1} \sigma, \varepsilon_t \varepsilon), \text{where } f_{p(t_1,t_2)}:=  (\Sigma_{t_1}, \Sigma_{t_2}, \sigma_t^{-1} \sigma, \varepsilon_t \varepsilon_t).
    \]
\end{enumerate}
\end{lemma}

For convenience, we give following notations where $n $ is nonnegative integer:
\begin{align*}
 \mathcal{P}_1(i_1,i_2,\cdots,i_n)(t_0) &:= \{ p(t_0, t) \in \mathcal{P}_{q}(i_1,i_2,\cdots,i_{n-1}) (t_0) \mid \mathbf{w}_{t_{\diamond_{k}}} (t_0, t) = i_k(k \in [n]) \} ,\\
H _1(i_1,i_2,\cdots,i_n)(t_0)  &:= \{ f_{p(t_0,t)} \mid p(t_0, t) \in \mathcal{P}_1(i_1,i_2,\cdots,i_n)(t_0)  \}, \\
G_1(i_1,i_2,\cdots,i_n)(t_0)   &:= \left\{ f \in \operatorname{Aut}(\mathcal{A}) \middle| \begin{array}{l} \exists p(t_0, t) \in \mathcal{P}_1(i_1,i_2,\cdots,i_n)(t_0), \sigma \in S_n , \\ \varepsilon \in \{\pm\} \text{ such that } f = (\Sigma_{t_0}, \Sigma_t, \sigma, \varepsilon) \end{array} \right\}.
\end{align*}

Clearly, $G_1(i_1,i_2,\cdots,i_{n+1})(t_0) \subset G_1(i_1,i_2,\cdots,i_n)(t_0)$.

\begin{lemma}\label{lemma111}
For any automorphism \( f \in G_{1}(i_1, \dots, i_{n+1})(t_0) \), \( f \) is the finite product of elements from sets \( G_0(t_0) \) and \( \bigcup_{\ell \in \{0,1,2\}} H_1(i_1, \dots, i_n, \ell)(t_0) \).
\end{lemma}

\begin{proof}
Let \( p(t_0, t) \in \mathcal{P}(i_1, \dots, i_n, i_{n+1})(t_0) \) and \( f = f_{p(t_0, t)} \). If \( i_{n+1} \in \{0,1\} \), the statement holds clearly. Assume that \( i_{n+1} = \ell \), the statement holds, then we consider \( i_{n+1} = \ell + 1 \).

Recall there exist \( \sigma_t \in S_n \) and \( \varepsilon_t \in \{\pm \} \) such that \( f = (\Sigma_{t_0}, \Sigma_t, \sigma_t, \varepsilon_t) \). Let \( t_i \) be the vertex on \( p(t_0, t) \) satisfying \( B_{t_i} \sim \pm B_{t_{n+1}} \), with:
\[
\begin{split}
\mathbf{w}_{t_0}(t_0, t_i) &= 2, \quad  \mathbf{w}_{t_{\diamond_{n+1}}}(t_0, t_i) = 2, \quad 0 \leq \mathbf{w}_{t_{\diamond_{k}}}(t_0, t_i) \leq i_k, (k \in [i]=\{0,1,\cdots,i\} ) \\
\mathbf{w}_{t_0}(t_i, t) &= 2, \quad \mathbf{w}_{t_{\diamond_{n+1}}}(t_i, t_{ }) = \ell, \quad 0 \leq \mathbf{w}_{t_{\diamond_{k}}}(t_i,t) \leq i_k.
\end{split}
\]

Suppose \( \tau(B_{t_{n+1}}) = \pm B_{t_s} \); set \( s = \mu_{p(t_{n+1}, t_0)}(t_i) \). By Lemma~\ref{rem:3.10}(1) and Assumption \( (\diamond_{n+1}) \), \( B_s \sim \pm B_{t_0} \), and \( w_{t_0}(t_i, s) = w_{t_{n+1}}(t_i, s) = 1 \), \( w_{t_k}(t_i, s) = 0 \) for all \( k \in [n] \). Consequently:
\[
\begin{split}
\mathbf{w}_{t_0}(t_0, s) &= 2, \quad 1 \leq \mathbf{w}_{t_{\diamond_{n+1}}}(t_0, s) \leq 2, \quad 0 \leq \mathbf{w}_{t_{\diamond_{k}}}(t_0, s) \leq i_k,  \\
\mathbf{w}_{t_0}(s, t) &= 2, \quad \ell - 1 \leq \mathbf{w}_{t_{\diamond_{n+1}}}(s, t) \leq \ell, \quad 0 \leq \mathbf{w}_{t_{\diamond_{k}}}(s,t) \leq i_k.
\end{split}
\]

By Lemma~\ref{rem:3.10}(2) and (3), \( f = h_{p(s, t)} \circ g_{p(t_0, s)} \). Clearly, \( g_{p(t_0, s)} \in G_1(i_1, \dots, i_n, 1)(t_0) \cup G_1(i_1, \dots, i_n, 2)(t_0) \), and \( h_{p(s, t)} \in G_1(i_1, \dots, i_n, \ell)(t_0) \cup G_1(i_1, \dots, i_n, \ell - 1)(t_0) \).

By Lemma~\ref{rem:3.10}(3), there exist \( \psi_g, \psi_h \in G_0(t_0) \),
\[
g \in H_1(i_1, \dots, i_n, 1)(t_0) \cup H_1(i_1, \dots, i_n, 2)(t_0),
\]
and
\[
h \in H_1(i_1, \dots, i_n, \ell)(t_0) \cup H_1(i_1, \dots, i_n, \ell - 1)(t_0)
\]
such that \( h_{p(s, t)} = h \circ g_h \) and \( g_{p(t_0, s)} = g \circ g_g \).

Thus, \( f = h \circ \psi_h \circ g \circ \psi_g \). By induction, we complete the proof.
\end{proof}

\begin{theorem}\label{thm:refined-generator00}
$\operatorname{Aut}(\mathcal{A})$ is generated by sets $G_0 (t_0)$ and $ \bigcup_{\substack{k \in [n] \\ i_k \in \{0,1,2\}}} H_{1}(i_1,i_2,\cdots,i_n)(t_0) $.
\end{theorem}
\begin{proof}
Firstly, we prove the case $n=2$.
For each $i\in\{0,1,2\}$, we have \[H_{1,i}^{t_\diamond}(t_0)=\bigcup_{k\ge 0}H_1(i,k)(t_0).\] 
Since \[H_1(i,k)(t_0)\subseteq G_1(i,k)(t_0),\] by Lemma~\ref{lemma111} with $n=1$,
every element of $G_1(i,k)(t_0)$ is a finite product of elements from 
\(G_0(t_0) \text{and} \bigcup_{j\in\{0,1,2\}}H_1(i,j)(t_0).\)
Hence $H_1(i,k)(t_0)$ is generated by 
\[G_0(t_0) \text{and} \bigcup_{j\in\{0,1,2\}}H_1(i,j)(t_0).\]

Since $\operatorname{Aut}(\mathcal{A})$ is generated by 
\[G_0(t_0) \text{and} \bigcup_{i\in\{0,1,2\}}H_{1,i}^{t_\diamond}(t_0),\]
it follows that $\operatorname{Aut}(\mathcal{A})$ is generated by 
\[G_0(t_0) \text{and} \bigcup_{i,j\in\{0,1,2\}}H_1(i,j)(t_0).\]

Then we assume the conclusion holds for \( n \geq 2 \), i.e., \( \operatorname{Aut}(\mathcal{A}) \) is generated by \[G_0(t_0) \ \text{  and} \bigcup_{\substack{k \in [n] \\ i_k \in \{0,1,2\}}} H_{1}(i_1,i_2,\cdots,i_n)(t_0). \] 

Note that for any $n>2$ and for any $i_1,\cdots,i_n$, \[H_1(i_1, \dots, i_n)(t_0) \subset G_1(i_1, \dots, i_n)(t_0), \] and \[ G_1(i_1, \dots, i_n)(t_0) = \bigcup_{\ell \geq 0} G_1(i_1, \dots, i_n, \ell)(t_0), \] and for each \( \ell \geq 0 \), \( G_1(i_1, \dots, i_n, \ell)(t_0) \) is generated by \[G_0(t_0) \ \text{ and } \bigcup_{\substack{k \in [n+1] \\ i_k \in \{0,1,2\}}} H_1(i_1, \dots, i_n, i_{n+1})(t_0), \] which follows from Lemma~\ref{lemma111}. 

Therefore, \(\bigcup_{\substack{k \in [n] \\  i_k \in \{0,1,2\}}} H_1(i_1, \dots, i_n)(t_0) \) is generated by the elements in \[ G_0(t_0) \ \text{  and} \bigcup_{\substack{k \in [n+1] \\ i_k \in \{0,1,2\}}} H_1(i_1, \dots, i_n, i_{n+1})(t_0), \] as desired.

\end{proof}

\subsubsection*{Marked vertices strategy to find generators of $\operatorname{Aut}(\mathcal{A})$ respectively.}

\begin{enumerate}[label={(\arabic*)}]  
    \item Fix marked vertices $t_\diamond$, $t_{\diamond_2}$, $\cdots$ and $t_{\diamond_n}$ satisfying Assumption$(\diamond_n)$.
    \item Find the vertex $t$ such that $B_t \thicksim\pm B_{t_0}$ or $B_t \sim \pm B_{t_{\diamond_k}}$ in each parth.
    \item This process stops if 
    \begin{enumerate}[label={(\roman*)}]  
        \item find another vertex $t$ with $B_t \sim \pm B_{t_0}$
        \item find vertices $t_1, t_2, t_3$ such that $B_{t_i} \sim \pm B_{t_{\diamond_k}}. ( i=1,2,3 \text{ and } k=1,2,\cdots,n) $
    \end{enumerate}
\end{enumerate}

\subsection{Algorithms}\label{abcd}

In this section, we give the algorithms for computing $\operatorname{Aut}(\mathcal{A})$.

Firstly, we design  Algorithm~\ref{ii} and Algorithm~\ref{iii} according to Theorem~\ref{thm:refined-generator00}.
Secondly, we reduce the number of generators by sign permutation equivalence(removing duplicate cluster automorphisms). Lastly, we simplify the generators and calculate the $\operatorname{Aut}(\mathcal{A})$.

\begin{algorithm}[H]
    \caption{Find the Set \( G_0 \)}\label{ii}
    \begin{algorithmic}[1]
        \Require \\
            Initial exchange matrix \( B_{t_0} \) of the cluster algebra,\\
            Permutation group \( S_n \)
        \Ensure 
            Set \( G_0 = \{ (\sigma, \pm) \mid \sigma(B_{t_0}) = \pm B_{t_0} \} \)

        \State Initialize \( G_0 \) as an empty set.

        \For{each permutation \( \sigma \in S_n \)}
            
            \If{ \( \sigma_(B_{t_0}) == B_{t_0} \) }
                \State Add \( (\sigma, +) \) to \( G_0 \).
            \ElsIf{ \( \sigma(B_{t_0}) == -B_{t_0} \) }
                \State Add \( (\sigma, -) \) to \( G_0 \).
            \EndIf
        \EndFor
        
        \State Output \( G_0 \).
    \end{algorithmic}
\end{algorithm}

\begin{algorithm}[H]
    For marked vertices \( t_{\diamond_1}, \dots, t_{\diamond_n} \), \( k_i \) is denoted the number of times the mutation path \( p \) passes through \( \pm \sigma(B_{\diamond}) \). In sequel, we call \( (k_1, \dots, k_n) \) the frequency vector of the path \( p \).

    \caption{Find the Set \( H_1 \)}\label{iii}
    \begin{algorithmic}[1]
        \Require \\
            Initial matrix \( B_{t_0} \),\\
            Permutation group \( S_n \),\\
            Mutation rules \( \{ \text{mutate}(\text{matrix}, k) \} \),\\
            Remarked matrices \( B_{t_{\diamond_1}}, B_{t_{\diamond_2}}, \cdots , B_{t_{\diamond_n}} \)
        \Ensure 
           \(H_1 (t_0) = \{ f_{p(t_0, t)} \mid p(t_0, t) \in \mathcal{P}_1 (t_0) \}\).

        \State Initialize a queue: \( \text{queue} = [ (B_{t_0}, [\text{  } ], (0, \cdots,0 )) ] \)  \hfill \Comment{[(current matrix,current mutation  path, frequency vector of the current path)]}
        \State \( \text{processed} = \emptyset \) 
        \State \( H_1 = \emptyset \)  \Comment{Initialize results}

        \While{ queue is not empty }
            \State Extract first entry: \( (\text{B}, \text{path}, (k_1, \cdots,k_n) \) from queue.

            \If{ \( k_i \geq 3 \) }
                \State Skip this path (terminate exploration of this path).
                \State \textbf{continue}
            \EndIf

            \For{mutation direction \( k \in n \)}
                \State Mutate \( \text{B} \) in direction \( k \): \( \text{new\_B} = \text{mutate}(\text{B}, k) \).
                \State Update path: \( \text{new\_path} = \text{path} + [k] \)  

                \If{ \( \text{length}(\text{new\_path}) > 1 \) and \( \text{new\_path}[-1] == \text{new\_path}[-2] \) }
                    \State \textbf{continue}
                \EndIf
                \State \( \text{new}\_k_i = k_i \).    \Comment{Initialize \( \text{new}\_k_i \)}
                \State \( \text{automorphism\_found} = \text{False} \)
                \For{each permutation \( \sigma_i \in S_n \)}
                    \State \( P = \sigma_i(\text{new\_B}) \)  \Comment{Apply permutation \( \sigma_i \)}
                    \If{ \( P == B_{t_0} \) }
                        \State Add \( (\text{new\_path}, \sigma_i, +, (\text{new}\_k_1, \cdots,\text{new}\_k_n)) \) to \( H_1 \).
                        \State \( \text{automorphism\_found} = \text{True} \)
                    \ElsIf{ \( P == -B_{t_0} \) }
                        \State Add \( (\text{new\_path}, \sigma_i, -, (\text{new}\_k_1,\cdots \text{new}\_k_n)) \) to \( H_1 \).
                        \State \( \text{automorphism\_found} = \text{True} \)
                    \ElsIf{ \( P == B_{t_{\diamond_i}} \text{ or } P == -B_{t_{\diamond_i}}\) } 
                        \State \( \text{new}\_k_i = k_i + 1 \)  
                    \EndIf
                \EndFor

                \If{\text{not } \( \text{automorphism\_found} \) }
                    \State Add \( (\text{new\_B}, \text{new\_path}, (\text{new}\_k_1, \cdots, \text{new}\_k_n)) \) to queue.
                \EndIf

                
            \EndFor
        \EndWhile

        \State Output \( H_1 \).
    \end{algorithmic}
\end{algorithm}

\begin{algorithm}[H]
    For any two generators $g_1=(\mathrm{path}_1,\sigma_1)$, $g_2=(\mathrm{path}_2,\sigma_2)$ in $H_1$, their composition is defined as
    \[
    g_1 \circ g_2 = \big( \mathrm{path}_1 + \sigma_1(\mathrm{path}_2),\ \sigma_1\circ\sigma_2 \big).
    \]
    A generator is redundant if its cluster coincides with the cluster obtained by composing other generators.

    \caption{Generator Redundancy Elimination via Composition}\label{alg:reduce}
    \begin{algorithmic}[1]
        \Require \\
            Generator set $H_1$,\\
            Initial cluster $X_{t_0}=(x_1,x_2,x_3,x_4)$,\\
            Cluster mutation function $\mathrm{mutate}(X, \mathrm{path})$,\\
            Identity generator $id$.
        Ensure 
            Irredundant generating set $H_1'$.

        \State $H_1' = H_1 \setminus \{ id \}$

        \ForAll{$g_1 \in H_1'$}
            \ForAll{$g_2 \in H_1'$}
                \State $\mathrm{path}_{\mathrm{new}} = \mathrm{path}_1 + \sigma_{1}(\mathrm{path}_2)$
                \State $\sigma_{\mathrm{new}} = \sigma_{1} \circ \sigma_{2}$
                
                \State $X_{\mathrm{new}} = \mathrm{mutate}(X_{t_0}, \mathrm{path}_{\mathrm{new}} )$
                
                \ForAll{$g_3 \in H_1'$}
                    \State $X_3 = \mathrm{mutate}(X_{t_0}, \mathrm{path}_3)$
                    \If{$X_{\mathrm{new}} == X_3$}
                        \State $H_1' = H_1' \setminus \{g\}$
                    \EndIf
                \EndFor
            \EndFor
        \EndFor
        

        \State Output $H_1'$.
    \end{algorithmic}
\end{algorithm}
\vspace{120pt}

\section{Applications in finite mutation type cluster algebra of rank 4}\label{s5}
In this section, we apply the algorithms developed in Section~\ref{abcd} to all cluster algebras of finite mutation type of rank $4$. Then we get the following three tables.

\medskip

For a rank $4$ cluster algebra $\mathcal{A}$ of finite mutation type, Table~\ref{tab:aut} lists the computational results (together with complexity) produced by the algorithms in Subsection~\ref{abcd}. The table presents the essential generators of each cluster automorphism group, the maximum traversal depth, and the number of mutation paths in $H_1$.

\begin{table}[H]
    \centering
    \small  
    \setlength{\tabcolsep}{3pt} 
    \begin{tabular}{>{\centering\arraybackslash}m{2cm}
                     >{\centering\arraybackslash}m{3cm}
                     >{\centering\arraybackslash}m{5.7cm}
                     >{\centering\arraybackslash}m{1.2cm}
                     >{\centering\arraybackslash}m{3cm}}
        \toprule
        \textbf{Diagram} & \textbf{Type}  & \textbf{Generators \allowbreak of $\operatorname{Aut}(\mathcal{A})$}  & \textbf{Max \allowbreak Depth} &\textbf{Mutation \allowbreak Paths  } \\
        \midrule
        \includegraphics[width=1.6cm]{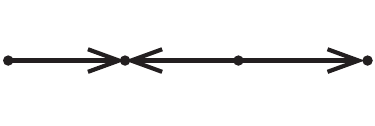} & $\mathbb{A}_4$  & $\psi_{(14)(23)}^{-}, g_{31}^{(14)(23), -}$ & 7&196 \\
        &&\\
        \includegraphics[width=1cm]{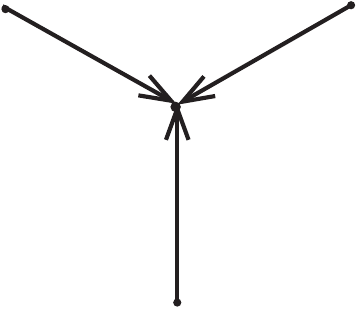} & $\mathbb{D}_4$  & $\psi_{(13)}^{+},\psi_{(34)}^{+},\psi_{(14)}^{+},g_{2}^{id, -}, g_{342321}^{(24), +}$ & 7&46\\
        \includegraphics[width=1.4cm]{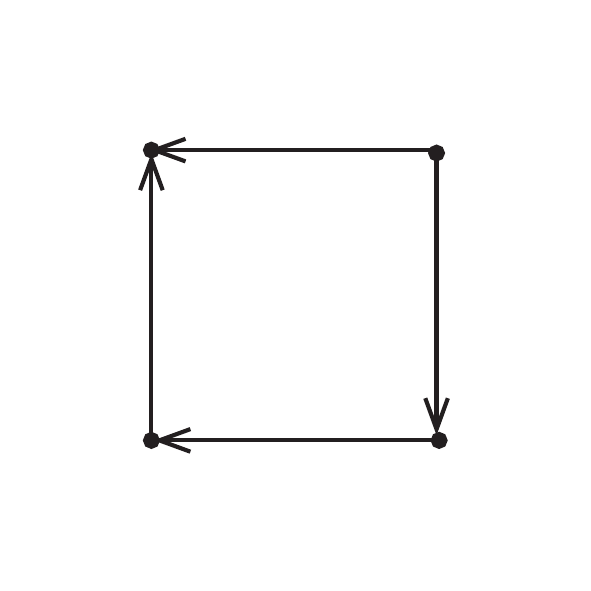} & $\tilde{\mathbb{A}}_{1,3}$  & $\psi_{(12)(34)}^{-}, g_1^{(24),-}$ & 5 &12\\
         \includegraphics[width=1.4cm]{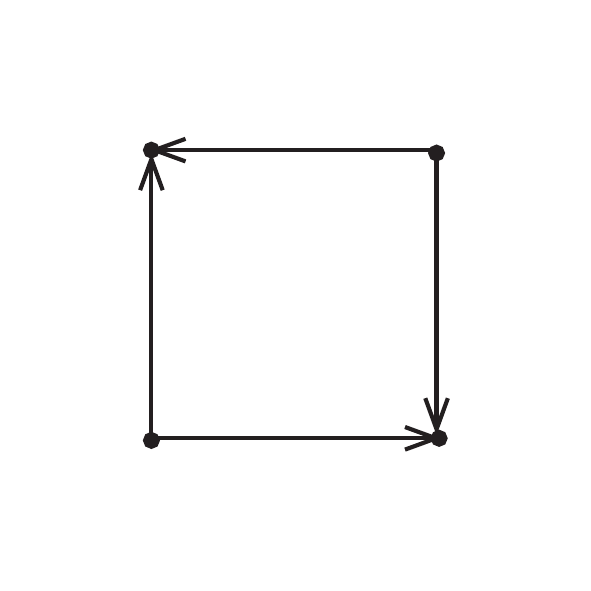} & $\tilde{\mathbb{A}}_{2,2}$ & $\psi_{(13)}^{+},\psi_{(4321)}^{-}, g_{31}^{(13),-}$ &7 &116\\
        \multirow{2}{1.6cm}{\includegraphics[width=1.6cm]{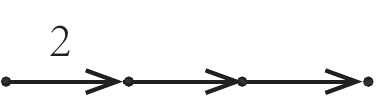}} 
            & $\mathbb{B}_4$ &$g_{31}^{id,-}, g_{42}^{id,-}$ & 13 &8356 \\
            & $\mathbb{C}_4$ &  $g_{31}^{id,-}, g_{42}^{id,-}$& 13&8356 \\
        &&\\
        \includegraphics[width=1.6cm]{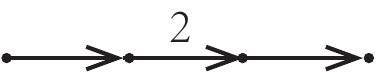}& $\mathbb{F}_4$  & $g_{31}^{id,-}, g_{42121}^{(12),-}$ & 15&13692\\
        \includegraphics[width=1.6cm]{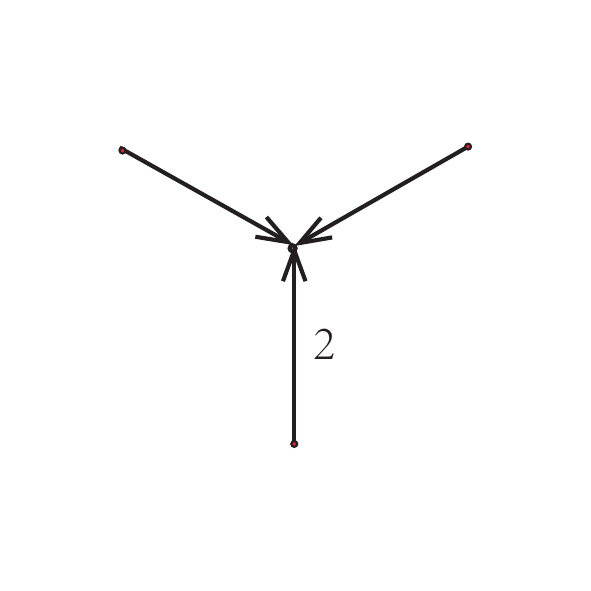} & $\mathbb{D}_{4}(2,1,1)$ & $\psi_{(34)}^{+}, g_{2}^{id,-}, g_{24231}^{(24),-}$ & 15&11218 \\
        \includegraphics[width=2cm]{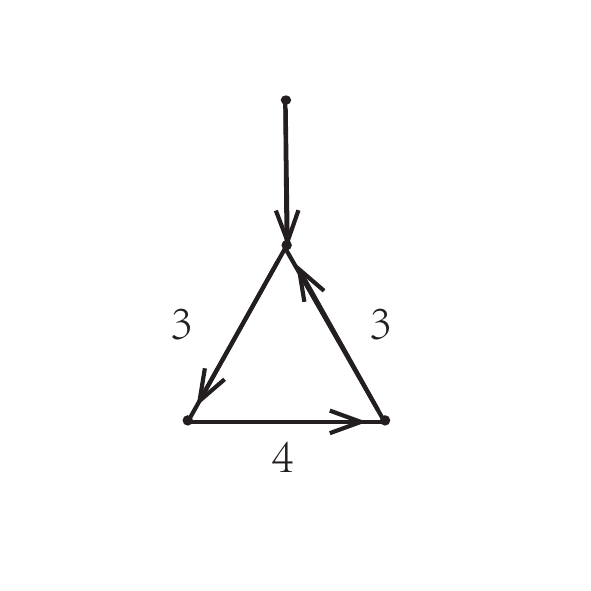}& $\mathbb{G}_{2}^{(*,+)}$  & $g_{1}^{(34),-}, g_{3}^{(34),+}, g_{1242}^{(12)(34),+}$ & 5&12\\
        \includegraphics[width=1.4cm]{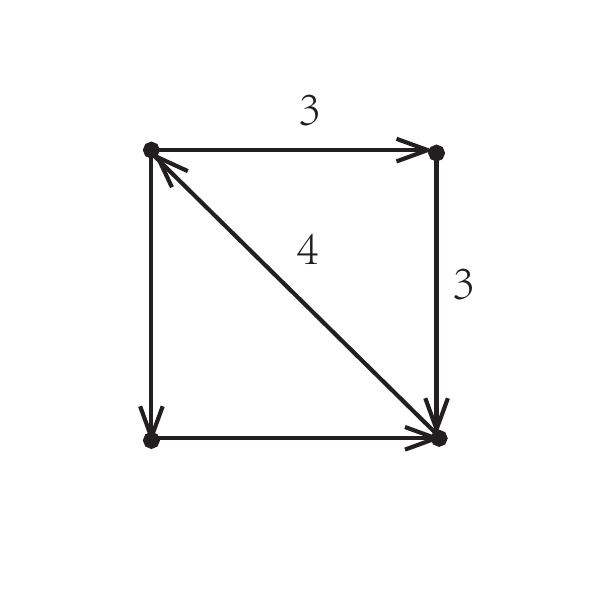}& $\mathbb{G}_{2}^{(*,*)}$  & $\psi_{(13)}^{-}, g_{1}^{id,-}, g_{42}^{id,-}, g_{212}^{(34),-}$&3&12\\
        \includegraphics[width=1.6cm]{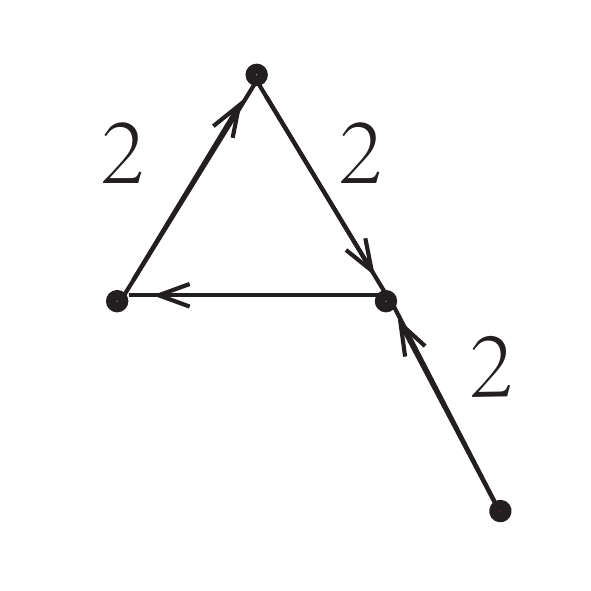}
                                & $(2,2,1,2)d_1 \neq d_3$  & $g_{31}^{id,-}, g_{43241}^{(24),+}$  &19&544724\\
        \includegraphics[width=1.6cm]{t5.pdf}                        & $(2,2,1,2)d_1 = d_3$  & $g_{2}^{(13),-}, g_{31}^{id,-}, g_{1241}^{(13)(24),+}$ & 11&358\\
                               
        \includegraphics[width=2cm]{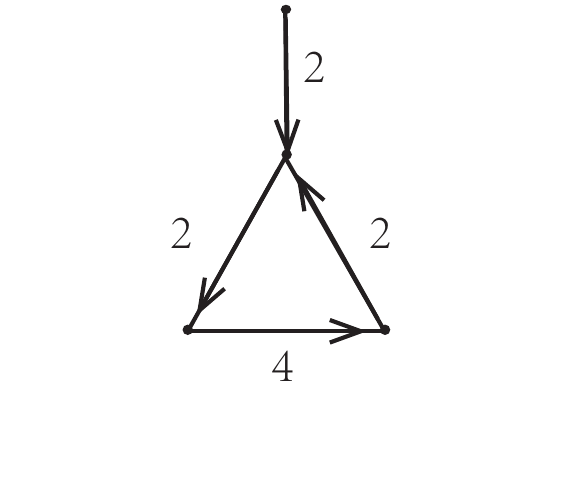}  & $(2,2,2,4)d_1=d_3$  & $g_{1}^{(34),-}, g_{3}^{(34),+}, g_{2}^{(13),-}$ & 1&4\\
        \includegraphics[width=2cm]{t7.pdf}                        & $(2,2,2,4)d_1 \neq d_3$ & $g_{1}^{(34),-}, g_{3}^{(34),+}, g_{2312}^{id,-}, g_{2432}^{id,-}$ & 5&18 \\
        
        \includegraphics[width=1cm]{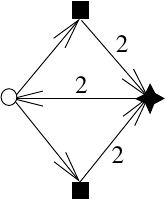} & $(1,1,2,2,2)$  & $g_{2}^{(13),-}, g_{31}^{(13),-}, g_{4213421}^{(14),-}, \psi_{(13)}^{+}$ & 15&8029\\
        \includegraphics[width=1cm]{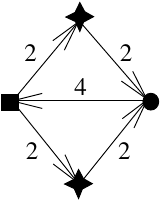} & $(2,2,2,2,4)$ & $\psi_{(13)}^{+}, \psi_{(24)}^{-}, g_{2}^{(24),+}, g_{31}^{(13),-}, g_{1421}^{(24),+}$ &5&41\\
        \bottomrule
    \end{tabular}
    \caption{Generators and data of cluster automorphism groups of finite mutation types.}
    \label{tab:aut}
\end{table}

\newpage

We recall some groups and group constructions that appear in the tables below:
\begin{itemize}
    \item $\{ id \}$: the trivial group.
    \item $D_m$: the dihedral group of order $2m$, which has a presentation as
    \[
    \langle \alpha, y \mid \alpha^m = y^2 = id, \, y\alpha y^{-1} = \alpha^{-1} \rangle.
    \]
    \item $D_\infty$: the infinite dihedral group, which has presentations as
    \[
    \langle r, s \mid r^2 = s^2 = id \rangle = \langle \alpha, y \mid y^2 = id, \, y\alpha y^{-1} = \alpha^{-1} \rangle,
    \]
    and the isomorphism is given by $r \mapsto \alpha y$ and $s \mapsto y$.
    \item $A \rtimes_{\rho} B$: the semidirect product of groups $A$ and $B$, where
    \(
    \rho: B \to \mathrm{Aut}(A).
    \)
    \item $A \times B$: the direct product of groups $A$ and $B$.
\end{itemize}

The essential generators listed in Table~\ref{tab:aut} are obtained by eliminating redundant generators from the full output of the algorithm. The complete sets of generators directly computed by the code is given in Table~\ref{tab:aut_generators} below.
\newcolumntype{M}[1]{>{\centering\arraybackslash}m{#1}}
\newcolumntype{L}[1]{>{\raggedright\arraybackslash}m{#1}}
\newcolumntype{R}[1]{>{\raggedleft\arraybackslash}m{#1}}

\newcolumntype{X}[1]{>{\raggedright\arraybackslash}m{#1}<{\hspace{0.4cm}}}

\normalsize 

\renewcommand{\arraystretch}{1.5} 

\begin{longtable}{M{2.5cm} M{4.0cm} >{\linespread{1.4}\selectfont}M{9.5cm}} 
    \toprule
    \textbf{Type} & \textbf{ Generators in $G_0$} & \textbf{Generators in $H_1$ } \\
    \midrule
    \endfirsthead
    
    \toprule
    \textbf{Type} & \textbf{Generators in $G_0$} & \textbf{Generators in $H_1$} \\
    \midrule
    \endhead
    
    \midrule
    \multicolumn{3}{r}{\footnotesize Continued on next page} \\
    \endfoot
    
    \bottomrule
    \multicolumn{3}{c}{\vspace{-0.2cm}} \\ 
    \caption{Generators with mutation paths of $\operatorname{Aut}(\mathcal{A})$ } \label{tab:aut_generators} \\
    \endlastfoot
    
    $\mathbb{A}_4$ & 
    $\psi_{id}^{+},\ \psi_{(14)(23)}^{-}$ & 
    $g_{31}^{(14)(23), -},\ g_{42121}^{(12), -},\ g_{42321}^{(23), +},$\par \vspace{0.12cm}
    $g_{143121}^{(12)(34), -},\ g_{343121}^{(12)(34), +},\ g_{214321}^{(241), -}$ \\
    \midrule

    $\mathbb{D}_4$ & 
    $\psi_{id}^{+},\ \psi_{(13)}^{+},\ \psi_{(34)}^{+},\ \psi_{(14)}^{+},$\par \vspace{0.12cm}
    $\psi_{(341)}^{+},\ \psi_{(431)}^{+}$ & 
    $g_2^{id,-},\ g_{431}^{id,-},\ g_{431321}^{(312),+},\ g_{341321}^{(2431),-},$\par \vspace{0.12cm}
    $g_{342321}^{(324),-},\ g_{214321}^{(12)(34),+},\ g_{124321}^{(34),-}$ \\
    \midrule

    $\tilde{\mathbb{A}}_{1,3}$ & 
    $\psi_{id}^{+},\ \psi_{(12)(34)}^{-}$ & 
    $g_1^{(24),-},\ g_2^{(13),-},\ g_{323}^{(14),-},\ g_{34134}^{(12),+}$ \\
    \midrule

    $\tilde{\mathbb{A}}_{2,2}$ & 
    $\psi_{id}^{+},\ \psi_{(13)}^{+},\ \psi_{(13)(24)}^{+},$
    $\psi_{(2341)}^{-},\ \psi_{(12)(34)}^{-},\ \psi_{(24)}^{+},$
    $\psi_{(14)(23)}^{-},\ \psi_{(4321)}^{-}$ &

    $ g_{42}^{(1432),+},\ g_{42321}^{(23),+},\ g_{31412}^{(1243),-},$\par \vspace{0.12cm}
    $g_{31}^{(13),-},\ g_{1423421}^{(13),-},\ g_{2314312}^{(1432),+}$ \\
    \midrule

    $\mathbb{B}_4$ & 
    $\{id\}$ & 
    $g_{31}^{id,-},\ g_{34123}^{(12),-},\ g_{42}^{id,-},\ g_{3412}^{id,+},$\par \vspace{0.12cm}
    $g_{2143}^{id,+},\ g_{413423}^{(312),+},\ g_{413421}^{(312),-},$\par \vspace{0.12cm}
    $g_{241234}^{(321),+},\ g_{2142341}^{(13),-}$ \\
    \midrule

    $\mathbb{C}_4$ & 
    $\{id\}$ & 
    $g_{42}^{id,-},\ g_{31}^{id,-},\ g_{3412}^{id,+},\ g_{2143}^{id,+},$\par \vspace{0.12cm}
    $g_{313423}^{(312),-},\ g_{413423}^{(312),+},\ g_{413421}^{(312),-},$\par \vspace{0.12cm}
    $g_{241234}^{(321),+},\ g_{2142341}^{(13),-}$ \\
    \midrule

    $\mathbb{F}_4$ ($d_2=2d_3$) & 
    $\{id\}$ & 
    $g_{31}^{id,-},\ g_{42121}^{(12),-},\ g_{24341}^{(34),+},$\par \vspace{0.12cm}
    $g_{143121}^{(12)(34),-},\ g_{343121}^{(12)(34),+}$ \\
    \midrule

    $\mathbb{F}_4$ ($2d_2=d_3$) & 
    $\{id\}$ & 
    $g_{31}^{id,-},\ g_{42121}^{(12),-},\ g_{24341}^{(34),+},$\par \vspace{0.12cm}
    $g_{143121}^{(12)(34),-},\ g_{343121}^{(12)(34),+}$ \\
    \midrule

    $\mathbb{D}_4(2,1,1)$ ($d_2=d_3=d_4$) & 
    $\psi_{(34)}^{+},\ \psi_{id}^{+}$ & 
    $g_{2}^{id,-},\ g_{24231}^{(24),-},\ g_{21423121}^{(24),+},$\par \vspace{0.12cm}
    $g_{142134123}^{(24),+},\ g_{4321413121}^{(23),-},\ g_{41213241321}^{id,-},$\par \vspace{0.12cm}
    $g_{32132413121}^{id,-}$ \\
    \midrule

    $\mathbb{D}_4(2,1,1)$ ($d_1=d_3=d_4$) & 
    $\psi_{(34)}^{+},\ \psi_{id}^{+}$ & 
    $g_{2}^{id,-},\ g_{24231}^{(24),-},\ g_{21423121}^{(24),+},$\par \vspace{0.12cm}
    $g_{142134123}^{(24),+},\ g_{4321413121}^{(23),-},\ g_{41213241321}^{id,-},$\par \vspace{0.12cm}
    $g_{32132413121}^{id,-}$ \\
    \midrule

    $\mathbb{G}_{2}^{(*,+)}$ & 
    $\{id\}$  & 
    $g_{1}^{(34),-},\ g_{3}^{(34),+},\ g_{4}^{(34),+},\ g_{1312}^{(12)(34),+},$\par \vspace{0.12cm}
    $g_{2412}^{id,-},\ g_{1242}^{(12)(34),+},\ g_{1342}^{(12)(34),-},\ g_{13132}^{(12),+},$\par \vspace{0.12cm}
    $g_{13232}^{(12),+}$ \\
    \midrule

    $\mathbb{G}_{2}^{(*,*)}$ & 
    $\psi_{(13)}^{-}$ & 
    $g_{1}^{id,-},\ g_{3}^{id,-},\ g_{42}^{id,-},\ g_{212}^{(34),-},$\par \vspace{0.12cm}
    $g_{312}^{(34),+},\ g_{132}^{(14),+},\ g_{232}^{(14),-},\ g_{214}^{(34),+},$\par \vspace{0.12cm}
    $g_{314}^{(34),-},\ g_{134}^{(14),-},\ g_{234}^{(14),+}$ \\
    \midrule

    $(2,2,1,2)$ ($d_1\neq d_3$) & 
    $\{id\}$  & 
    $g_{31}^{id,-},\ g_{2312}^{id,-},\ g_{43241}^{(24),+},\ g_{4324}^{(24),+},$\par \vspace{0.12cm}
    $g_{32412}^{(24),+},\ g_{321321}^{id,-},\ g_{21243121}^{id,-}$ \\
    \midrule

    $(2,2,1,2)$ ($d_1=d_3$) & 
    $\{id\}$  & 
    $g_{2}^{(13),-},\ g_{31}^{id,-},\ g_{1241}^{(13)(24),+},\ g_{12321}^{(13),-},$\par \vspace{0.12cm}
    $g_{423142121}^{(13),-},\ g_{3241}^{(13)(24),-},\ g_{32121}^{(13),+},\ g_{4324}^{id,-},$\par \vspace{0.12cm}
    $g_{123243124}^{(13),+},\ g_{232342123}^{(24),-},\ g_{3243}^{(13)(24),-},$\par \vspace{0.12cm}
    $g_{1243}^{(13)(24),-},\ g_{2142123}^{(24),+}$ \\
    \midrule

    $(2,2,2,4)$ ($d_1=d_3$) & 
    $\{id\}$  & 
    $g_{1}^{(34),-},\ g_{2}^{(13),-},\ g_{3}^{(34),+},\ g_{4}^{(34),+}$ \\
    \midrule

    $(2,2,2,4)$ ($d_1\neq d_3$) & 
    $\{id\}$  & 
    $g_{1}^{(34),-},\ g_{3}^{(34),+},\ g_{4}^{(34),+},\ g_{2312}^{id,-},$\par \vspace{0.12cm}
    $g_{2432}^{id,-},\ g_{2142}^{id,+},\ g_{2412}^{id,+},\ g_{23212}^{(34),-},$\par \vspace{0.12cm}
    $g_{24212}^{(34),+},\ g_{21232}^{(34),+}$ \\
    \midrule

    $(1,1,2,2,2)$ & 
    $\psi_{(13)}^{+},\ \psi_{id}^{+}$ & 
    $g_{2}^{(13),-},\ g_{31}^{(13),-},\ g_{4213421}^{(14),-},\ g_{12123}^{(13),+},$\par \vspace{0.12cm}
    $g_{41234121}^{(14),+},\ g_{4324234121}^{(13),-},\ g_{134214}^{(34),+},$\par \vspace{0.12cm}
    $g_{31234121}^{(14),-}$ \\
    \midrule

    $(2,2,2,2,4)$ & 
    $\psi_{(13)}^{+},\ \psi_{(24)}^{-},\ \psi_{(13)(24)}^{-}$ & 
    $g_{2}^{(24),+},\ g_{4}^{(24),+},\ g_{31}^{(13),-},\ g_{1421}^{(24),+},$\par \vspace{0.12cm}
    $g_{1423}^{(13)(24),-},\ g_{12321}^{(24),+},\ g_{14341}^{(24),+}$ \\
\end{longtable}

We further analyze the algebraic relations among the essential generators and identify the dihedral subgroups within each cluster automorphism group. The results are summarized in Table~\ref{tab:cluster_aut}.
\begin{longtable}{X{2.2cm} >{\linespread{1.3}\selectfont}L{5.2cm} L{4.8cm} M{3.8cm}}
    \toprule
    \textbf{Type} & \textbf{Generators of $\operatorname{Aut}(\mathcal{A})$} & \textbf{Relations} & \textbf{Dihedral subgroup} \\
    \midrule
    \endfirsthead
    
    \toprule
    \textbf{Type} & \textbf{ enerators of $\operatorname{Aut}(\mathcal{A})$} & \textbf{Relations} & \textbf{Dihedral subgroup} \\
    \midrule
    \endhead
    
    \midrule
    \multicolumn{4}{r}{\footnotesize Continued on next page} \\
    \endfoot
    
    \bottomrule
    \multicolumn{4}{c}{\vspace{-0.2cm}} \\ 
    \caption{Cluster automorphism groups: generators, relations, and dihedral subgroups} \label{tab:cluster_aut} \\
    \endlastfoot

    $\mathbb{A}_4$ & 
    $x = \psi_{(14)(23)}^{-},\; y = g_{31}^{(14)(23), -}$ & 
    $x^2 = y^7 = id$, \par \vspace{0.12cm} $yx = xy^{-1}$ & 
    $\operatorname{Aut}(\mathcal{A}) \cong D_7$ \\
    \midrule

    $\mathbb{D}_4$ & 
    $x = g_{2}^{id, -},\; y = g_{342321}^{(24), +}$, \par\vspace{0.10cm} 
    $S_3$: $\psi_{(13)}^{+},\;\psi_{(34)}^{+},\;\psi_{(14)}^{+}$ & 
    $x^2 = y^4 = id$, \par \vspace{0.12cm} $yx = xy^{-1}$ & 
    $\operatorname{Aut}(\mathcal{A}) \cong S_3 \times D_4$ \\
    \midrule

    $\tilde{\mathbb{A}}_{1,3}$ & 
    $x = \psi_{(12)(34)}^{-},\; y = g_1^{(24),-}$ & 
    $x^2 = y^2 = id$ & 
    $\operatorname{Aut}(\mathcal{A}) \cong D_\infty$ \\
    \midrule

    $\tilde{\mathbb{A}}_{2,2}$ & 
    $a = \psi_{(13)}^{+},\; b = \psi_{(4321)}^{-}$, \par\vspace{0.10cm} 
    $c = g_{31}^{(13),-}$ & 
    $a^2 = b^4 = c^2 = id$, \par \vspace{0.12cm} $ac = ca$, $ab^{-1} = ba$, \par \vspace{0.12cm} $c b^2 = b^2 c$ & 
    $\langle g_{31}^{(13),-},\psi_{(14)(23)}^{-} \rangle \cong D_\infty$ \\
    \midrule

    $\mathbb{B}_{4}$ & 
    $x = g_{31}^{id,-},\; y = g_{42}^{id,-}$ & 
    $x^2 = y^5 = id$, \par \vspace{0.12cm} $yx = xy^{-1}$ & 
    $\operatorname{Aut}(\mathcal{A}) \cong D_5$ \\
    \midrule

    $\mathbb{C}_{4}$ & 
    $x = g_{31}^{id,-},\; y = g_{42}^{id,-}$ & 
    $x^2 = y^5 = id$, \par \vspace{0.12cm} $yx = xy^{-1}$ & 
    $\operatorname{Aut}(\mathcal{A}) \cong D_5$ \\
    \midrule

    $\mathbb{F}_{4}$ & 
    $x = g_{31}^{id,-},\; y = g_{42121}^{(12),-}$ & 
    $x^2 = y^7 = id$, \par \vspace{0.12cm} $yx = xy^{-1}$ & 
    $\operatorname{Aut}(\mathcal{A}) \cong D_7$ \\
    \midrule

    $\mathbb{D}_{4}(2,1,1)$ & 
    $z = \psi_{(34)}^{+},\; r = g_{2}^{id,-},\;$\par \vspace{0.12cm}  $s =g_{2}^{id,-} \circ g_{24231}^{(24),-}$ & 
    $r^2 = s^2 = z^2 = id$, \par \vspace{0.12cm} $rz = zr$, $sz = zs$, \par \vspace{0.12cm} $\langle r,s \rangle \cong D_{\infty}$ & 
    $\operatorname{Aut}(\mathcal{A}) \cong D_\infty \times Z_2$ \\
    \midrule

    $\mathbb{G}_{2}^{(*,+)}$ & 
    $x = g_{1}^{(34),-},\; y = g_{3}^{(34),+},\;$\par \vspace{0.12cm} $z = g_{1242}^{(12)(34),+}$ & 
    $z^9 = x^2 = (yx)^2 = id$, \par \vspace{0.12cm} $z^3 y = y z^3$ & 
    $\langle g_{1}^{(34),-}, g_{3}^{(34),+} \rangle \cong D_\infty$ \\
    \midrule

    $\mathbb{G}_{2}^{(*,*)}$ & 
    $x = \psi_{(13)}^{-},\; a = g_{1}^{id,-},\;$\par \vspace{0.12cm}$ y = g_{42}^{id,-},\; b = g_{212}^{(34),-}$ & 
    $x^2 = a^2 = y^2 = b^2 = id$, \par \vspace{0.12cm} $(by)^6 = id$, \par \vspace{0.12cm} $a(by)^3 = (by)^3 a$, \par \vspace{0.12cm} $xy = yx$, $ab = ba$ & 
    $\langle \psi_{(13)}^{-}, g_{1}^{id,-} \rangle \cong D_\infty$ \\
    \midrule

    $(2,2,1,2)$ \par $d_1 \neq d_3$ & 
    $r = g_{31}^{id,-},\; s = g_{43241}^{(24),+} \circ g_{31}^{id,-} $ & 
    $r^2 = s^2 = id$ & 
    $\operatorname{Aut}(\mathcal{A}) \cong D_{\infty}$ \\
    \midrule

    $(2,2,1,2)$ \par $d_1 = d_3$ & 
    $x = g_{2}^{(13),-},\; y = g_{31}^{id,-},\;$\par \vspace{0.12cm}$ z = g_{1241}^{(13)(24),+}$ & 
    $x^2 = y^2 = id$, \par \vspace{0.12cm} $z^{-1} x z = x$, \par \vspace{0.12cm} $z^{-1} y z = y$ & 
    $\operatorname{Aut}(\mathcal{A}) \cong Z \rtimes_\rho D_\infty$ \\
    \midrule

    $(2,2,2,4)$ \par $d_1 = d_3$ & 
    $x = g_{2}^{(13),-},\; y = g_{3}^{(34),+}$, \par\vspace{0.10cm} 
      & 
    $x^2 = (x y x y^{-1} x)^2 = id$, \par \vspace{0.12cm} $(y x y x y^{-1} x)^2 = id$ & 
    $\langle g_{21}^{(143),+},g_{2}^{(13),-} \rangle \cong D_\infty$ \\
    \midrule

    $(2,2,2,4)$ \par $d_1 \neq d_3$ & 
    $a = g_{1}^{(34),-},\; b = g_{13}^{id,-} $, \par\vspace{0.10cm} 
    $c = g_{2312}^{id,-},  g_{2432}^{id,-}$ & 
    $a^2 = b^2 = id$, \par \vspace{0.12cm} $(a c a c b)^2 = id$ & 
    $\langle g_{13}^{id,-}, g_{2312}^{id,-} \rangle \cong D_\infty$ \\
    \midrule

    $(1,1,2,2,2)$ & 
    $a = g_{2}^{(13),-},\; b = g_{31}^{(13),-}$, \par\vspace{0.10cm} 
    $c = g_{4213421}^{(14),-},\; z = \psi_{(13)}^{+}$ & 
    $a^2 = b^2 = c^2 = z^2 = id$, \par \vspace{0.12cm} $az = za$, $bz = zb$, $cz = zc$ & 
    $\langle g_{2}^{(13),-}, g_{31}^{(13),-} \rangle \cong D_\infty$ \\
    \midrule

    $(2,2,2,2,4)$ & 
    $x = \psi_{(13)}^{+},\; y = \psi_{(24)}^{-},\;$\par \vspace{0.12cm}$ a = \psi_{(24)}^{-} \circ g_{2}^{(24),+}$, \par\vspace{0.10cm} 
    $b = g_{31}^{(13),-},\; c =\psi_{(13)}^{+} \circ g_{1421}^{(24),+}$ & 
    $x^2 = y^2 = a^2 = b^2 = id$, \par \vspace{0.12cm} $c^4 = (x c)^2 = id$, \par \vspace{0.12cm} $xy = yx$, $xa = ax$, $xb = bx$, \par \vspace{0.12cm} $yb = by$, $yc = cy$, $ca = ac$, \par \vspace{0.12cm} $bc = cb$ & 
    $\langle \psi_{(24)}^{-},\; g_{2}^{(24),+} \rangle \cong D_\infty$ \\
\end{longtable}

We obtain the following theorem from the analysis of Table~\ref{tab:cluster_aut}.

\begin{theorem}\label{okk}
Suppose that $\mathcal{A}$ is a cluster algebra of finite mutation type of rank $4$ and $\operatorname{Aut}(\mathcal{A})$ is the cluster automorphism group of $\mathcal{A}$. Then
\begin{enumerate}[label=(\arabic*)]
    \item $\operatorname{Aut}(\mathcal{A})$ is a finite group if and only if $\mathcal{A}$ is of finite type.
    \item $\operatorname{Aut}(\mathcal{A})$ has a subgroup isomorphic to a dihedral group.
\end{enumerate}
\end{theorem}

\appendix

\section*{Appendix}
\addcontentsline{toc}{section}{Appendix}

\section{Classification of finite mutation type of rank 4}
\label{s3}

\subsection*{Skew-symmetric matrices}  
In Figure~\ref{xieduicheng}, we present all diagrams of skew-symmetric matrices of finite mutation types of rank 4 cf.\cite{fomin2007cluster111}.  
\begin{figure}[H]
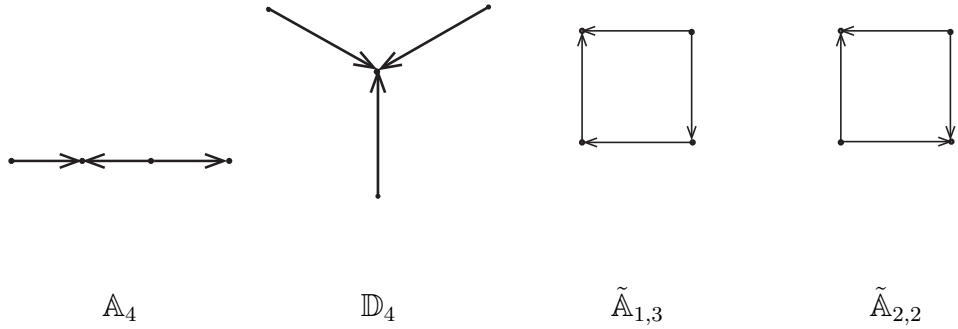

\centering
\setlength{\abovecaptionskip}{6pt}
\setlength{\belowcaptionskip}{6pt}

\vspace{6mm}
\begin{tabular}{cccc}
\includegraphics[width=3cm]{a4.pdf}
&
\includegraphics[width=3cm]{d4.pdf}
&
\includegraphics[width=3cm]{a13.pdf}
&
\includegraphics[width=3cm]{a22.pdf}
\\
\\
\vspace{6mm}
$\mathbb{A}_4$& $\mathbb{D}_4$ &$\tilde{\mathbb{A}}_{1,3}$&$\tilde{\mathbb{A}}_{2,2}$ \\
\end{tabular}
\caption{Diagrams of skew-symmetric matrices of rank 4.}
\label{xieduicheng}
\vspace{5mm}
\end{figure}

\subsection*{Not skew-symmetric matrices}
In this subsection, we present the classification of finite mutation types of rank 4 with skew-symmetrizable but not skew-symmetric matrices via s-decomposition in~\cite{felikson2012cluster}. 

\begin{definition}
A \textbf{block} is a diagram isomorphic to one of the diagrams with black/white colored vertices shown in Table~\ref{tab:all-blocks-3x5}, or to a single vertex. Vertices marked in white are called \textbf{outlets}, we call the black ones \textbf{dead ends}.
\end{definition}

\begin{definition}
A connected diagram $S$ is called \textbf{s-decomposable} if it can be obtained from a collection of blocks by identifying outlets of different blocks along some partial matching (matching of outlets of the same block is not allowed), where two single edges with same endpoints and opposite directions cancel out, and two single edges with same endpoints and same directions form an edge of weight $4$. 

A non-connected diagram $S$ is called s-decomposable either if $S$ satisfies the definition above, or if $S$ is a disjoint union of several diagrams (without any edge joining one to another) satisfying the definition above.
 
If a diagram $S$ is not s-decomposable then we call $S$ \textbf{non-decomposable}. 
\end{definition}

\begin{remark}
We call a vertex $u \in S$ an \emph{outlet} of $S$ if $S$ is block-decomposable and there exists a block decomposition of $S$ such that $u$ is contained in exactly one block $B$, and $u$ is an outlet of $B$.
\end{remark}


\begin{table}[H]  
\centering
\setlength{\abovecaptionskip}{6pt}   
\setlength{\belowcaptionskip}{6pt}   
\vspace{6mm} 

\begin{tabular}{|c|c|c|c|c|}
\hline
\includegraphics[width=1.9cm]{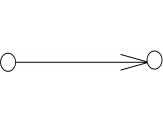}
& \includegraphics[width=1.9cm]{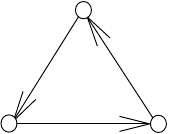}
& \includegraphics[width=1.9cm]{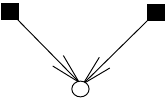}
& \includegraphics[width=1.9cm]{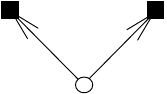}
& \includegraphics[width=1.9cm]{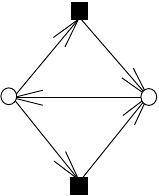}
\\[1mm]  
\hline
\includegraphics[width=1.9cm]{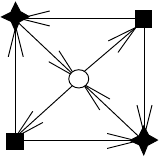}
& \includegraphics[width=1.9cm]{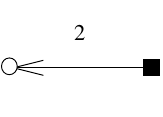}
& \includegraphics[width=1.9cm]{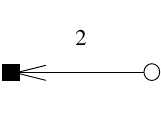}
& \includegraphics[width=1.9cm]{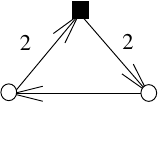}
& \includegraphics[width=1.9cm]{block51t.pdf}
\\[1mm]
\hline
\includegraphics[width=1.9cm]{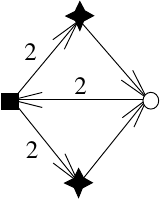}
& \includegraphics[width=1.9cm]{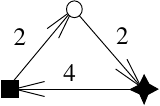}
& \includegraphics[width=1.9cm]{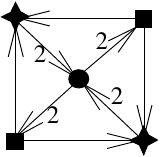}
& \includegraphics[width=1.9cm]{block6t2.pdf}
& \includegraphics[width=1.9cm]{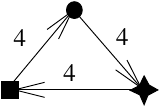}
\\[1mm]
\hline
\end{tabular}

\caption{Blocks.(Outlets are colored white, dead ends are black.)}
\label{tab:all-blocks-3x5}  
\vspace{6mm}
\end{table}

The following are all diagrams of finite mutation exchange matrices of rank 4 which is skew-symmetrizable but not skew-symmetric in Figure~\ref{teshutu}, cf.\cite{felikson2012cluster}.

\begin{figure}[H]
\centering
\setlength{\abovecaptionskip}{6pt}
\setlength{\belowcaptionskip}{6pt}
\vspace{6mm}

\begin{tabular}{ccc}

\includegraphics[width=3cm]{g2.pdf}
&
\includegraphics[width=3cm]{g21.pdf}
&
\includegraphics[width=3cm]{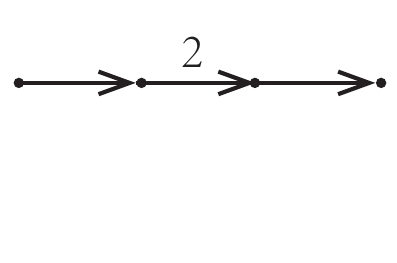}
\\  
\vspace{10mm}
$G(*,*)$& $G(*,+)$ &$F_4$ \\
\end{tabular}
\caption{Exceptional diagrams of rank 4. }
\label{teshutu}
\vspace{6mm}
\end{figure}

In Table~\ref{jjjsblocks}, we present all s-decomposable diagrams of rank 4. Furthermore, up to sign permutation equivalence, we select one representative from each mutation-equivalent class and mark it in red.

\begin{table}[H]
\centering
\setlength{\abovecaptionskip}{10pt}
\setlength{\belowcaptionskip}{10pt}

\vspace{-50pt}
\caption{S-decomposable diagrams.} 
\label{jjjsblocks}

\begin{tabular}{l}

\includegraphics[width=12.8cm]{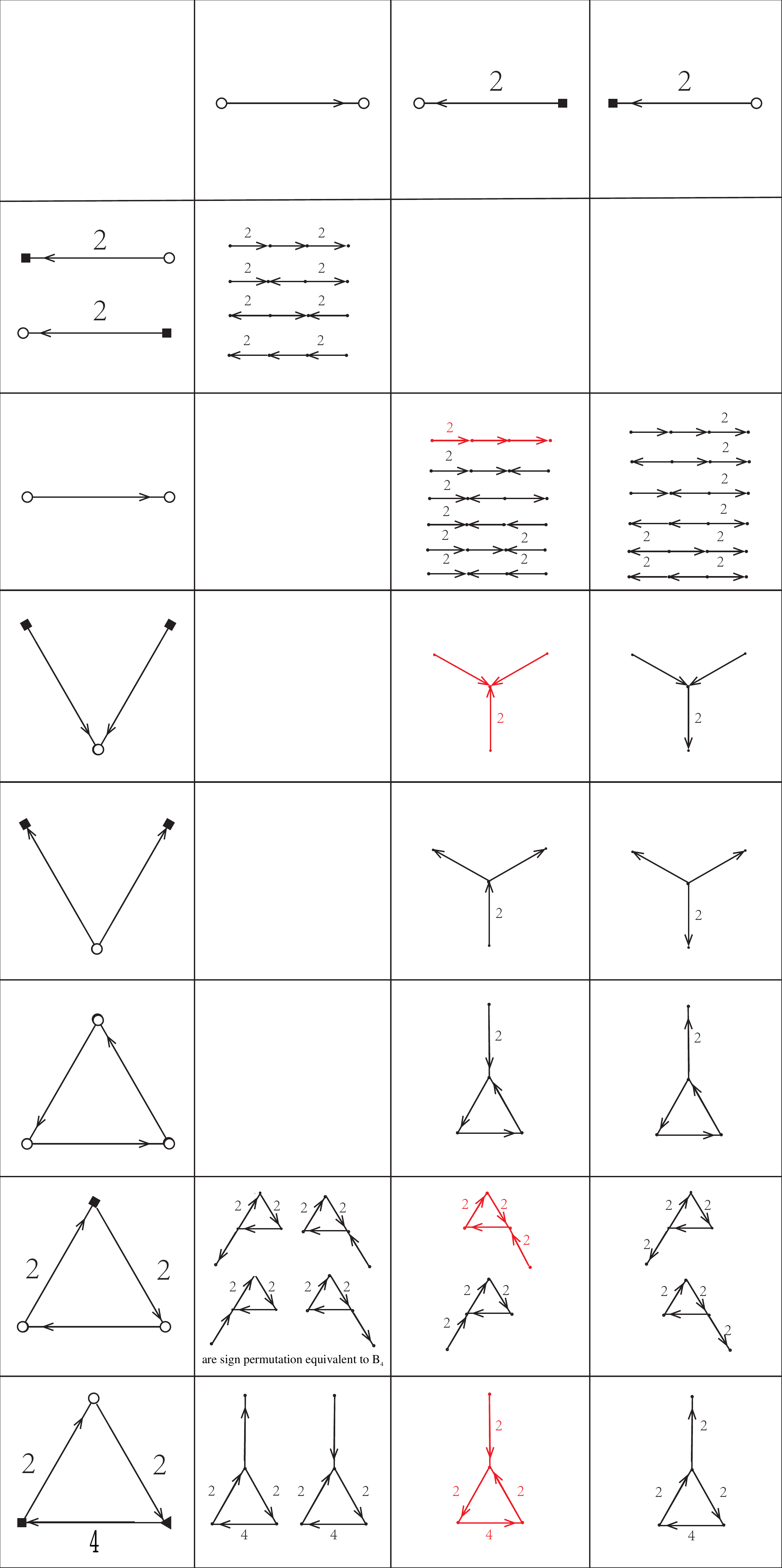}
\\ 

\end{tabular}

\end{table}

\section{Proofs of the Infinite Order of Generators}\label{app:inf}
Firstly, we recall the definition of $\mathbf{d}$-vectors, which were introduced by Fomin and Zelevinsky in \cite{FominZelevinsky2002}.

\begin{definition}
    Let $\Sigma$ be a cluster pattern of rank $n$. For any cluster variable $x_{i;t}$, denote by $\mathbf{d}_{i;t} = (d_1, \ldots, d_n)$ the integer vector such that $-d_k$ is the lowest degree of $x_{k;t_0}$ in the Laurent polynomial expression of $x_{i;t}$ in $\mathbf{x}_{t_0}$. We call $\mathbf{d}_{i;t}$ the denominator vector of $x_{i;t}$ with respect to $t_0$.
\end{definition}

The following recursion formula for $\mathbf{d}$-vectors is useful.

\begin{lemma}
     Let $\Sigma$ be a cluster pattern of rank $n$. For any pair of vertices $t, t' \in \mathbb{T}_n$ that are $k$-adjacent, the d-vectors of them satisfy the following relation:
\[
\mathbf{d}_{k;t'} = -\mathbf{d}_{k;t} + \max\left\{ \sum_{i=1}^n [b_{ik;t}]_+ \mathbf{d}_{i;t}, \sum_{i=1}^n [-b_{ik;t}]_+ \mathbf{d}_{i;t} \right\},
\]
where $\max\{\alpha, \beta\} = (\max\{a_1, b_1\}, \ldots, \max\{a_n, b_n\})$ for $\alpha = (a_1, \ldots, a_n)$ and $\beta = (b_1, \ldots, b_n)$.
\end{lemma}

\subsection*{$\tilde{\mathbb{A}}_{1,3}$ type}\label{app:A13}
We prove that the order of $g_1^{(24),-} \circ \psi_{(12)(34)}^{-} = g_1^{(1432),+}$ is infinite. For any $n \in Z^+, k \geq 0,$ we denote 
\[
t_n=
\begin{cases} 
(\mu_{2341})^k(t_0), &n=4k\\   
\mu_1(\mu_{2341})^k(t_0), &n=4k+1\\
\mu_{41}(\mu_{2341})^k(t_0), &n=4k+2\\
\mu_{341}(\mu_{2341})^k(t_0), &n=4k+3
\end{cases}
\]
Clearly, $\mu_{2341}(B_{t_0})=B_{t_0}$, which means the exchange matrices consist a cycle of length 4. We have
\begin{align*}
\mathbf{d}_{1;t_{n+1}} &= -\mathbf{d}_{1;t_n} + \max\{(\mathbf{d}_{2;t_{n}} + \mathbf{d}_{4;t_n}),\mathbf{0}\},  &&\mathbf{d}_{j;t_{n+1}}=\mathbf{d}_{j;t_{n}},   &&j\neq 1,n=4k,\\
\mathbf{d}_{4;t_{n+1}} &= -\mathbf{d}_{4;t_n} + \max\{(\mathbf{d}_{1;t_n} + \mathbf{d}_{3;t_n}),\mathbf{0}\},    &&\mathbf{d}_{j;t_{n+1}}=\mathbf{d}_{j;t_{n}},  &&j\neq 4,n=4k+1, \\
\mathbf{d}_{3;t_{n+1}} &= -\mathbf{d}_{3;t_n} + \max\{(\mathbf{d}_{2;t_n} + \mathbf{d}_{4;t_n}),\mathbf{0}\},    &&\mathbf{d}_{j;t_{n+1}}=\mathbf{d}_{j;t_{n}},  &&j\neq 3,n=4k+2,\\
\mathbf{d}_{2;t_{n+1}} &= -\mathbf{d}_{2;t_n} + \max\{(\mathbf{d}_{1;t_n} + \mathbf{d}_{3;t_n}),\mathbf{0}\},    &&\mathbf{d}_{j;t_{n+1}}=\mathbf{d}_{j;t_{n}},  &&j\neq 2,n=4k+3.
\end{align*}

We claim that for any $n \in Z^+, k \geq 0$, we have 
\begin{equation}
\begin{aligned}
\mathbf{d}_{2;t_{n}} &\geq \underset{i \in\{1,2,3,4\}}{\max}\{ \mathbf{d}_{i;t_{n}} \}, &&n=4k,\\
\mathbf{d}_{1;t_{n}} &\geq \underset{i \in\{1,2,3,4\}}{\max}\{ \mathbf{d}_{i;t_{n}} \}, &&n=4k+1, \\
\mathbf{d}_{4;t_{n}} &\geq \underset{i \in\{1,2,3,4\}}{\max}\{ \mathbf{d}_{i;t_{n}} \}, &&n=4k+2,\\
\mathbf{d}_{3;t_{n}} &\geq \underset{i \in\{1,2,3,4\}}{\max}\{ \mathbf{d}_{i;t_{n}} \}, &&n=4k+3,\\
\mathbf{d}_{i;t_{n}} &\geq \mathbf{0}, &&i\in\{1,2,3,4\}, n\geq 4 .
\end{aligned}
\label{eq:1}
\end{equation}

By a direct calculation, 
\begin{align*}
\mathbf{d}_{1;t_{1}} &= -\mathbf{d}_{1;t_0} + \max\{(\mathbf{d}_{2;t_{0}} + \mathbf{d}_{4;t_0}),\mathbf{0}\}=(1,0,0,0),\\
\mathbf{d}_{4;t_{2}} &= -\mathbf{d}_{4;t_1} + \max\{(\mathbf{d}_{1;t_1} + \mathbf{d}_{3;t_1}),\mathbf{0}\}=(1,0,0,1), \\
\mathbf{d}_{3;t_{3}} &= -\mathbf{d}_{3;t_2} + \max\{(\mathbf{d}_{2;t_2} + \mathbf{d}_{4;t_2}),\mathbf{0}\}=(1,0,1,1),\\
\mathbf{d}_{2;t_{4}} &= -\mathbf{d}_{2;t_3} + \max\{(\mathbf{d}_{1;t_3} + \mathbf{d}_{3;t_3}),\mathbf{0}\}=(2,1,1,1).
\end{align*}
and
\begin{align*}
&\mathbf{d}_{1;t_{1}}=(1,0,0,0),
&&\mathbf{d}_{2;t_{1}}=(0,-1,0,0),
&&\mathbf{d}_{3;t_{1}}=(0,0,-1,0),
&&\mathbf{d}_{4;t_{1}}=(0,0,0,-1), \\
&\mathbf{d}_{1;t_{2}}=(1,0,0,0),
&&\mathbf{d}_{2;t_{2}}=(0,-1,0,0),
&&\mathbf{d}_{3;t_{2}}=(0,0,-1,0),
&&\mathbf{d}_{4;t_{2}}=(1,0,0,1), \\
&\mathbf{d}_{1;t_{3}}=(1,0,0,0),
&&\mathbf{d}_{2;t_{3}}=(0,-1,0,0),
&&\mathbf{d}_{3;t_{3}}=(1,0,1,1),
&&\mathbf{d}_{4;t_{3}}=(1,0,0,1),\\
&\mathbf{d}_{1;t_{4}}=(1,0,0,0),
&&\mathbf{d}_{2;t_{4}}=(2,1,1,1),
&&\mathbf{d}_{3;t_{4}}=(1,0,1,1),
&&\mathbf{d}_{4;t_{4}}=(1,0,0,1).
\end{align*}
the above inequalities~\eqref{eq:1} hold for $n \leq 4$. We induct on $n$ and assume the statement holds for $ n>4$,  if $n=4k+1$, $\mathbf{d}_{4;t_{n+1}} = -\mathbf{d}_{4;t_n} + (\mathbf{d}_{1;t_n} + \mathbf{d}_{3;t_n}),$ and $ \mathbf{d}_{3;t_{n}}=\mathbf{d}_{3;t_{n-2}} \geq \mathbf{d}_{4;t_{n-2}}= \mathbf{d}_{4;t_{n}}$, so $ \mathbf{d}_{4;t_{n+1}} \geq \mathbf{d}_{1;t_{n}} \geq \underset{i \in\{1,2,3,4\}}{\max}\{ \mathbf{d}_{i;t_{n}} \} \geq \mathbf{0},$ which means $\mathbf{d}_{4;t_{n+1}} \geq \underset{i \in\{1,2,3,4\}}{\max}\{\mathbf{d}_{i;t_{n+1}}\} \geq \mathbf{0}.$ And the other cases are similar.  

It's easy to see that for any $n \in Z^+$ and $k \in \{1,2,3,4\}$, we have $\mathbf{d}_{k;t_{n+1}} \geq  \mathbf{d}_{k;t_n}.$ So, $(g_1^{(1432),+})^n = id$ if and only if $n=0$. And there is no $\sigma \neq id,\sigma \in S_4$ and $n\in Z^+$, let $sgn(\sigma)$ be the sign of $\sigma$, such that $(g_1^{(1432),+})^n= \psi_{\sigma}^{sgn(\sigma)}$. 

It's follows that for any $i \in Z$ and $j \in\{0, 1\}$ such that $(g_1^{(1432),+})^i \circ (\psi_{(12)(34)}^{-})^j = id$ if and only if $i=j=0$. Hence the order of $g_1^{(1432),+}$ is infinite.

\subsection*{$\tilde{\mathbb{A}}_{2,2}$ type}\label{app:A22}
We prove that $cba=g_{31}^{(4321),-}$ has infinite order. From $\mu_{42132431}(B_{t_0})=(B_{t_0})$, the exchange matrices consist a cycle of length 8. 
For any $n \in Z^+, k \geq 0,$ we denote 
\[
t_n=
\begin{cases} 
(\mu_{42132431})^k(t_0), &n=10k\\   
\mu_1(\mu_{42132431})^k(t_0), &n=10k+1\\
\mu_{31}(\mu_{42132431})^k(t_0), &n=10k+2\\
\mu_{431}(\mu_{42132431})^k(t_0), &n=10k+3\\
\mu_{2431}(\mu_{42132431})^k(t_0), &n=10k+4\\
\mu_{32431}(\mu_{42132431})^k(t_0), &n=10k+5\\  
\mu_{132431}(\mu_{42132431})^k(t_0), &n=10k+6\\
\mu_{2132431}(\mu_{42132431})^k(t_0), &n=10k+7
\end{cases}
\]
We claim that for any $n \geq 8$, we have 
\begin{equation}
\begin{aligned}
&\mathbf{d}_{i;t_{n+8k}} -\mathbf{d}_{i;t_{n}} = k(2,2,2,2),&&i\in \{ 1,2,3,4\}, k\geq 1.
\end{aligned}
\label{eq:0-0}
\end{equation}
The proof of the relations~\eqref{eq:0-0} is similar to Section~\ref{app:A13}. Consequently, the order of $cba$ is infinite.

\subsection*{$\mathbb{D}_{4}(2,1,1)$ type}\label{app:D4_211}
We prove that the order of $g_{24231}^{(24),-}$ is infinite. For any $n \in Z^+, k \geq 0,$ we denote 
\[
t_n=
\begin{cases} 
(\mu_{4243124231})^k(t_0), &n=10k\\   
\mu_1(\mu_{4243124231})^k(t_0), &n=10k+1\\
\mu_{31}(\mu_{4243124231})^k(t_0), &n=10k+2\\
\mu_{231}(\mu_{4243124231})^k(t_0), &n=10k+3\\
\mu_{4231}(\mu_{4243124231})^k(t_0), &n=10k+4\\
\mu_{24231}(\mu_{4243124231})^k(t_0), &n=10k+5\\  
\mu_{124231}(\mu_{4243124231})^k(t_0), &n=10k+6\\
\mu_{3124231}(\mu_{4243124231})^k(t_0), &n=10k+7\\
\mu_{43124231}(\mu_{4243124231})^k(t_0), &n=10k+8\\
\mu_{243124231}(\mu_{4243124231})^k(t_0), &n=10k+9
\end{cases}
\]
From $\mu_{24231}(B_{t_0})=(24)(B_{t_0})$, the exchange matrices consist a cycle of length 10. We have
\begin{equation*}
\begin{aligned}
\mathbf{d}_{1;t_{n+1}} &= -\mathbf{d}_{1;t_n} + \max\{\mathbf{d}_{2;t_{n}} ,\mathbf{0}\},\quad  \mathbf{d}_{j;t_{n+1}}=\mathbf{d}_{j;t_{n}},   &&j\neq 1,n=10k,\\
\mathbf{d}_{3;t_{n+1}} &= -\mathbf{d}_{3;t_n} + \max\{\mathbf{d}_{2;t_n} ,\mathbf{0}\}, \quad  \mathbf{d}_{j;t_{n+1}}=\mathbf{d}_{j;t_{n}},   &&j\neq 3,n=10k+1, \\
\mathbf{d}_{2;t_{n+1}} &= -\mathbf{d}_{2;t_n} + \max\{(2\mathbf{d}_{1;t_n} + \mathbf{d}_{3;t_n}),\mathbf{d}_{4;t_n}\}, \quad  \mathbf{d}_{j;t_{n+1}}=\mathbf{d}_{j;t_{n}},   &&j\neq 2,n=10k+2,\\
\mathbf{d}_{4;t_{n+1}} &= -\mathbf{d}_{4;t_n} + \max\{(2\mathbf{d}_{1;t_n} + \mathbf{d}_{3;t_n}),\mathbf{d}_{2;t_n}\}, \quad  \mathbf{d}_{j;t_{n+1}}=\mathbf{d}_{j;t_{n}},   &&j\neq 4,n=10k+3,\\
\mathbf{d}_{2;t_{n+1}} &= -\mathbf{d}_{2;t_n} + \max\{\mathbf{d}_{4;t_n},\mathbf{0}\}, \quad  \mathbf{d}_{j;t_{n+1}}=\mathbf{d}_{j;t_{n}},   &&j\neq 2,n=10k+4,\\
\mathbf{d}_{1;t_{n+1}} &= -\mathbf{d}_{1;t_n} + \max\{\mathbf{d}_{4;t_{n}} ,\mathbf{0}\}, \quad  \mathbf{d}_{j;t_{n+1}}=\mathbf{d}_{j;t_{n}},   &&j\neq 1,n=10k+5,\\
\mathbf{d}_{3;t_{n+1}} &= -\mathbf{d}_{3;t_n} + \max\{\mathbf{d}_{4;t_n} ,\mathbf{0}\}, \quad  \mathbf{d}_{j;t_{n+1}}=\mathbf{d}_{j;t_{n}},   &&j\neq 3,n=10k+6, 
\end{aligned}
\end{equation*}
\begin{equation}
\begin{aligned}
\mathbf{d}_{4;t_{n+1}} &= -\mathbf{d}_{4;t_n} + \max\{(2\mathbf{d}_{1;t_n} + \mathbf{d}_{3;t_n}),\mathbf{d}_{2;t_n}\}, \quad  \mathbf{d}_{j;t_{n+1}}=\mathbf{d}_{j;t_{n}},   &&j\neq 4,n=10k+7,\\
\mathbf{d}_{2;t_{n+1}} &= -\mathbf{d}_{2;t_n} + \max\{(2\mathbf{d}_{1;t_n} + \mathbf{d}_{3;t_n}),\mathbf{d}_{4;t_n}\}, \quad  \mathbf{d}_{j;t_{n+1}}=\mathbf{d}_{j;t_{n}},   &&j\neq 2,n=10k+8,\\
\mathbf{d}_{4;t_{n+1}} &= -\mathbf{d}_{4;t_n} + \max\{\mathbf{d}_{2;t_n},\mathbf{0}\},\quad  \mathbf{d}_{j;t_{n+1}}=\mathbf{d}_{j;t_{n}},   &&j\neq 4,n=10k+9.
\end{aligned}
\label{eq:bzd}
\end{equation}

We claim that for any $n \geq 10$, we have 
\begin{equation}
\begin{aligned}
&\mathbf{d}_{i;t_{n}} \geq \mathbf{0}, &&i \in \{1,2,3,4\},\\
&(2\mathbf{d}_{1;t_n} + \mathbf{d}_{3;t_n}) \geq \mathbf{d}_{4;t_n}, \\
&(2\mathbf{d}_{1;t_n} + \mathbf{d}_{3;t_n}) \geq\mathbf{d}_{2;t_n}, \\
&\mathbf{d}_{2;t_{n+10k}} -\mathbf{d}_{2;t_{n}} = k(-4\mathbf{d}_{1;t_n}+ 4\mathbf{d}_{2;t_n} -2\mathbf{d}_{3;t_n} -2\mathbf{d}_{4;t_n}),&&k\geq 1,n \in 10Z^+,\\
&\mathbf{d}_{4;t_{n+10k}} -\mathbf{d}_{4;t_{n}} = k(-4\mathbf{d}_{1;t_n}+ 4\mathbf{d}_{2;t_n} -2\mathbf{d}_{3;t_n} -2\mathbf{d}_{4;t_n}),&&k\geq 1,n \in \{5,15,25,35\cdots\}.
\end{aligned}
\label{eq:2}
\end{equation}
Indeed, $n=[10,20] \cap Z$, the above relations~\eqref{eq:2} hold.

Firstly, we prove the inequalities in~\eqref{eq:2}, we induct on $n$ and assume the statement holds for $ n\geq 20$, if $n=10k+2$, we have $(2\mathbf{d}_{1;t_{n+1}} + \mathbf{d}_{3;t_{n+1}})=(2\mathbf{d}_{1;t_{n}} + \mathbf{d}_{3;t_{n}}) \geq -\mathbf{d}_{2;t_n} + (2\mathbf{d}_{1;t_n} + \mathbf{d}_{3;t_n})=\mathbf{d}_{2;t_{n+1}}  \geq \mathbf{0}$ and $ \mathbf{d}_{i;t_{n}} \geq \mathbf{0}, (i \in \{1,3,4\})$. The other cases are similar.

Then, we prove the equalities in~\eqref{eq:2}.
By~\eqref{eq:bzd}, for any $n\in 10Z^+$, we have 
\begin{align*}
\mathbf{d}_{1;t_{n+10}} &=\mathbf{d}_{1;t_{n+6}}  \\
                            &=-\mathbf{d}_{1;t_{n+5}} + \mathbf{d}_{4;t_{n+5}} \\
                            &=-\mathbf{d}_{1;t_{n+1}} +\mathbf{d}_{4;t_{n+4}} \\
                            &=-(-\mathbf{d}_{1;t_{n}} +\mathbf{d}_{2;t_{n}} )+(-\mathbf{d}_{4;t_{n+3}} +2\mathbf{d}_{1;t_{n+3}}+\mathbf{d}_{3;t_{n+3}})\\
                            &=\mathbf{d}_{1;t_{n}} -\mathbf{d}_{2;t_{n}} -\mathbf{d}_{4;t_{n}} +2(-\mathbf{d}_{1;t_{n}} +\mathbf{d}_{2;t_{n}})+\mathbf{d}_{3;t_{n+2}}\\
                            &=-\mathbf{d}_{1;t_{n}} +\mathbf{d}_{2;t_{n}} -\mathbf{d}_{4;t_{n}} +(-\mathbf{d}_{3;t_{n+1}}+\mathbf{d}_{2;t_{n+1}})\\
                            &=-\mathbf{d}_{1;t_{n}} +\mathbf{d}_{2;t_{n}} -\mathbf{d}_{4;t_{n}} -\mathbf{d}_{3;t_{n}}+\mathbf{d}_{2;t_{n}}\\                            
                            &=-\mathbf{d}_{1;t_{n}} +2\mathbf{d}_{2;t_{n}} -\mathbf{d}_{3;t_{n}}-\mathbf{d}_{4;t_{n}}.
\end{align*}
Similarly, if $n\in 10Z^+$, the following equations hold:
\begin{equation}
\begin{aligned}
\mathbf{d}_{1;t_{n+10}} &= -\mathbf{d}_{1;t_n} +2\mathbf{d}_{2;t_{n}} -\mathbf{d}_{3;t_{n}}-\mathbf{d}_{4;t_{n}},\\
\mathbf{d}_{2;t_{n+10}} &= -4\mathbf{d}_{1;t_n} +5\mathbf{d}_{2;t_{n}} -2\mathbf{d}_{3;t_{n}}-2\mathbf{d}_{4;t_{n}},\\
\mathbf{d}_{3;t_{n+10}} &= -2\mathbf{d}_{1;t_n} +2\mathbf{d}_{2;t_{n}} -\mathbf{d}_{4;t_{n}},\\
\mathbf{d}_{4;t_{n+10}} &= -2\mathbf{d}_{1;t_n} +2\mathbf{d}_{2;t_{n}} -\mathbf{d}_{3;t_{n}}.
\end{aligned}
\label{eq:bzd2}
\end{equation}
By~\eqref{eq:bzd2}, for any $n\in 10Z^+, k\geq 1$, we have
 \[\mathbf{d}_{2;t_{n+10k}}-\mathbf{d}_{2;t_{n+10k-10}} = -4\mathbf{d}_{1;t_n} +4\mathbf{d}_{2;t_{n}} -2\mathbf{d}_{3;t_{n}}-2\mathbf{d}_{4;t_{n}}.\]               
What's more,
\begin{align*}
&\quad-4\mathbf{d}_{1;t_{n+10}} +4\mathbf{d}_{2;t_{n+10}} -2\mathbf{d}_{3;t_{n+10}}-2\mathbf{d}_{4;t_{n+10}}\\
&=-4(-\mathbf{d}_{1;t_n} +2\mathbf{d}_{2;t_{n}} -\mathbf{d}_{3;t_{n}}-\mathbf{d}_{4;t_{n}}) +4( -4\mathbf{d}_{1;t_n} +5\mathbf{d}_{2;t_{n}} -2\mathbf{d}_{3;t_{n}}-2\mathbf{d}_{4;t_{n}})\\
&\quad -2( -2\mathbf{d}_{1;t_n} +2\mathbf{d}_{2;t_{n}} -\mathbf{d}_{4;t_{n}}) -2( -2\mathbf{d}_{1;t_n} +2\mathbf{d}_{2;t_{n}} -\mathbf{d}_{3;t_{n}})\\
&=-4\mathbf{d}_{1;t_n}+4\mathbf{d}_{2;t_n}-2\mathbf{d}_{3;t_{n}}-2\mathbf{d}_{4;t_{n}}\\
                                                                                                      &=\cdots(\text{by recursion})\\
                                                                                                      &=-4\mathbf{d}_{1;t_{10}}+4\mathbf{d}_{2;t_{10}}-2\mathbf{d}_{3;t_{10}}-2\mathbf{d}_{4;t_{10}}\\
                                                                                                      &=(4,4,2,2).
\end{align*}

So, if $k\geq 1,n \in 10Z^+$, we have $\mathbf{d}_{2;t_{n+10k}} -\mathbf{d}_{2;t_{n+10k-10}} = k(-4\mathbf{d}_{1;t_n} +4\mathbf{d}_{2;t_n} -2\mathbf{d}_{3;t_n} -2\mathbf{d}_{4;t_n})$. Furthermore, if $k\geq 1$, we have $\mathbf{d}_{2;t_{10k}} -\mathbf{d}_{2;t_{10}} =(k-1)(-4\mathbf{d}_{1;t_{10}}+4\mathbf{d}_{2;t_{10}} -2\mathbf{d}_{3;t_{10}}  -2\mathbf{d}_{4;t_{10}})=(k-1)(4,4,2,2)$. 

Thus, $(g_{24231}^{(24),-})^n(x_2) = id(x_2)$ if and only if $n=0$, which means $(g_{24231}^{(24),-})^n = id$ if and only if $n=0$. Similarly, $(g_{24231}^{(24),-})^n(x_4)=x_4$ if and only if $n=0$.

What's more, ($(g_{24231}^{(24),-})^i \circ (g_2^{id,-})^j)(x_4)=(g_{24231}^{(24),-})^i (x_4)$. 
It's follows that for any $i \in Z$ and $j \in\{0, 1\}$ such that $(g_{24231}^{(24),-})^i \circ (g_2^{id,-})^j = id$ if and only if $i=j=0$. Hence the order of $g_{24231}^{(24),-}$ is infinite.

\subsection*{$\mathbb{G}_{2}^{(*,+)}$ type}\label{app:G2+}
We prove that the order of $y=g_{3}^{(34),+}$ is infinite. For any $n \in Z^+, k \geq 0,$ we denote 
\[
t_n=
\begin{cases} 
(\mu_{43})^k(t_0), &n=2k\\   
\mu_3(\mu_{43})^k(t_0), &n=2k+1
\end{cases}
\]
Clearly, $\mu_{43}(B_{t_0})=B_{t_0}$, which means the exchange matrices consist a cycle of length 2. We have
\begin{align*}
\mathbf{d}_{3;t_{n+1}} &= -\mathbf{d}_{3;t_n} + \max\{\mathbf{d}_{2;t_{n}}, 2\mathbf{d}_{4;t_n}\}, \quad  \mathbf{d}_{j;t_{n+1}}=\mathbf{d}_{j;t_{n}},   &&j\neq 3,n=2k,\\
\mathbf{d}_{4;t_{n+1}} &= -\mathbf{d}_{4;t_n} + \max\{\mathbf{d}_{2;t_n},2\mathbf{d}_{3;t_n}\}, \quad  \mathbf{d}_{j;t_{n+1}}=\mathbf{d}_{j;t_{n}},   &&j\neq 4,n=2k+1.
\end{align*}
And for any $n > 2$, it's clearly that $\mathbf{d}_{i;t_n}=\mathbf{d}_{i;t_0}\leq \mathbf{0},i\in \{1,2\}$.
And we claim that for any $n \in Z^+,k \geq 0$, we have 
\begin{equation}
\begin{aligned}
\mathbf{d}_{4;t_{n}} &\geq \underset{i \in\{1,2,3,4\}}{\max}\{ \mathbf{d}_{i;t_{n}} \}, &&n=2k,\\
\mathbf{d}_{3;t_{n}} &\geq \underset{i \in\{1,2,3,4\}}{\max}\{ \mathbf{d}_{i;t_{n}} \}, &&n=2k+1,\\
\mathbf{d}_{i;t_{n}} &\geq \mathbf{0} &&i\in\{3,4\},n\geq 2.
\end{aligned}
\label{eq:000}
\end{equation}
Indeed, $n=1,2$, the above inequalities~\eqref{eq:000} hold. We induct on $n$ and assume the statement holds for $ n=2k$. For the case $n+1$, $\mathbf{d}_{3;t_{n+1}} = -\mathbf{d}_{3;t_n} + 2\mathbf{d}_{4;t_n},$ and $\mathbf{d}_{4;t_{n}} \geq \mathbf{d}_{3;t_{n}}$, so $\mathbf{d}_{3;t_{n+1}}=(\mathbf{d}_{4;t_{n}} -\mathbf{d}_{3;t_{n}}) +\mathbf{d}_{4;t_{n}}  \geq \mathbf{d}_{4;t_{n}} =\mathbf{d}_{4;t_{n+1}} \geq \underset{i \in\{1,2,3,4\}}{\max}\{ \mathbf{d}_{i;t_{n}} \},$ then $\mathbf{d}_{3;t_{n+1}} \geq \underset{i \in\{1,2,3,4\}}{\max}\{ \mathbf{d}_{i;t_{n+1}} \}.$ And the other case is similar.  

It's easy to see that for any $n \in Z^+$ and $k \in \{1,2,3,4\}$, we have $\mathbf{d}_{k;t_{n+1}} \geq  \mathbf{d}_{k;t_n}.$ So, $y^n = id$ if and only if $n=0$, which means the order of $y$ is infinite.

\subsection*{$(2,2,1,2)$ type, $d_1 \neq d_3$}\label{app:2212_1}
We prove that the order of $g_{43241}^{(24),+}$ is infinite when $d_1=2d_2$. For any $n \in Z^+,k\geq 0,$ we denote 
\[
t_n=
\begin{cases} 
(\mu_{2342143241})^k(t_0), &n=10k\\   
\mu_1(\mu_{2342143241})^k(t_0), &n=10k+1\\
\mu_{41}(\mu_{2342143241})^k(t_0), &n=10k+2\\
\mu_{241}(\mu_{2342143241})^k(t_0), &n=10k+3\\
\mu_{3241}(\mu_{2342143241})^k(t_0), &n=10k+4\\
\mu_{43241}(\mu_{2342143241})^k(t_0), &n=10k+5\\
\mu_{143241}(\mu_{2342143241})^k(t_0), &n=10k+6\\
\mu_{2143241}(\mu_{2342143241})^k(t_0), &n=10k+7\\
\mu_{42143241}(\mu_{2342143241})^k(t_0), &n=10k+8\\
\mu_{342143241}(\mu_{2342143241})^k(t_0), &n=10k+9
\end{cases}
\]
From $\mu_{24231}(B_{t_0})=(24)(B_{t_0})$, the exchange matrices consist a cycle of length 10. We have
\begin{align*}
\mathbf{d}_{1;t_{n+1}} &= -\mathbf{d}_{1;t_n} + \max\{2\mathbf{d}_{2;t_{n}} ,\mathbf{0}\},\quad  \mathbf{d}_{j;t_{n+1}}=\mathbf{d}_{j;t_{n}},   &&j\neq 1,n=10k,\\
\mathbf{d}_{4;t_{n+1}} &= -\mathbf{d}_{4;t_n} + \max\{\mathbf{d}_{2;t_n} ,2\mathbf{d}_{3;t_n}\}, \quad  \mathbf{d}_{j;t_{n+1}}=\mathbf{d}_{j;t_{n}},   &&j\neq 4,n=10k+1, \\
\mathbf{d}_{2;t_{n+1}} &= -\mathbf{d}_{2;t_n} + \max\{\mathbf{d}_{1;t_n},\mathbf{d}_{4;t_n}\}, \quad  \mathbf{d}_{j;t_{n+1}}=\mathbf{d}_{j;t_{n}},   &&j\neq 2,n=10k+2,\\
\mathbf{d}_{3;t_{n+1}} &= -\mathbf{d}_{3;t_n} + \max\{\mathbf{d}_{4;t_n} ,\mathbf{0}\}, \quad  \mathbf{d}_{j;t_{n+1}}=\mathbf{d}_{j;t_{n}},   &&j\neq 3,n=10k+3,\\
\mathbf{d}_{4;t_{n+1}} &= -\mathbf{d}_{4;t_n} + \max\{(2\mathbf{d}_{1;t_n} + \mathbf{d}_{3;t_n}),\mathbf{d}_{2;t_n}\}, \quad  \mathbf{d}_{j;t_{n+1}}=\mathbf{d}_{j;t_{n}},   &&j\neq 4,n=10k+4,\\
\mathbf{d}_{1;t_{n+1}} &= -\mathbf{d}_{1;t_n} + \max\{2\mathbf{d}_{4;t_{n}} ,\mathbf{0}\}, \quad  \mathbf{d}_{j;t_{n+1}}=\mathbf{d}_{j;t_{n}},   &&j\neq 1,n=10k+5,\\
\mathbf{d}_{2;t_{n+1}} &= -\mathbf{d}_{2;t_n} + \max\{\mathbf{d}_{4;t_n} ,2\mathbf{d}_{3;t_n}\}, \quad  \mathbf{d}_{j;t_{n+1}}=\mathbf{d}_{j;t_{n}},   &&j\neq 2,n=10k+6, \\
\mathbf{d}_{4;t_{n+1}} &= -\mathbf{d}_{4;t_n} + \max\{\mathbf{d}_{1;t_n},\mathbf{d}_{2;t_n}\}, \quad  \mathbf{d}_{j;t_{n+1}}=\mathbf{d}_{j;t_{n}},   &&j\neq 4,n=10k+7,\\
\mathbf{d}_{3;t_{n+1}} &= -\mathbf{d}_{3;t_n} + \max\{\mathbf{d}_{2;t_n} ,\mathbf{0}\}, \quad  \mathbf{d}_{j;t_{n+1}}=\mathbf{d}_{j;t_{n}},   &&j\neq 3,n=10k+8,\\
\mathbf{d}_{2;t_{n+1}} &= -\mathbf{d}_{2;t_n} + \max\{(2\mathbf{d}_{1;t_n} + \mathbf{d}_{3;t_n}),\mathbf{d}_{4;t_n}\}, \quad  \mathbf{d}_{j;t_{n+1}}=\mathbf{d}_{j;t_{n}},   &&j\neq 2,n=10k+9.
\end{align*}
Similar to Section~\ref{app:A13}, the following relations hold:
\begin{equation}
\begin{aligned}
\mathbf{d}_{4;t_{n+30}} &= \mathbf{d}_{4;t_n} ,&&n \in 10Z^+,\\
\mathbf{d}_{4;t_{n+10}} &\geq \mathbf{d}_{4;t_n} ,&&n \in \{15,25,35,\cdots\}.
\end{aligned}
\label{eq:4}
\end{equation}
Thus, $(g_{43241}^{(24),+})^n = id$ if and only if $n=0$, and $(g_{43241}^{(24),+})^n(x_4)=x_4$ if and only if $n=0$.
What's more, ($(g_{43241}^{(24),+})^i \circ (g_{31}^{id,-})^j)(x_4)=(g_{43241}^{(24),+})^i (x_4)$. 
It's follows that for any $i \in Z$ and $j \in\{0, 1\}$ such that $(g_{43241}^{(24),+})^i \circ (g_{31}^{id,-})^j = id$ if and only if $i=j=0$. Hence the order is infinite.

\subsection*{$(2,2,1,2)$ type, $d_1 = d_3$}\label{app:2212_2}
Similar to Section~\ref{app:A13}, the order of $g_{2}^{(13),-} \circ g_{31}^{id,-}$ is infinite. (The detailed d-vector analysis is analogous to the previous cases and is omitted.)

\subsection*{$(2,2,2,4)$ type, $d_1=d_3$}\label{app:2224_1}
We prove that the orders of $y$, $xyxy^{-1}$, and $xyxyxy^{-1}x$ are infinite.

First, $y=g_{3}^{(34),+}$: For any $n \in Z^+,k\geq 0$ we denote 
\[
t_n=
\begin{cases} 
(\mu_{43})^k(t_0), &n=2k\\   
\mu_3(\mu_{43})^k(t_0), &n=2k+1
\end{cases}
\]
From $\mu_{43}(B_{t_0})=B_{t_0}$, the exchange matrices consist a cycle of length 2. We have
\begin{align*}
\mathbf{d}_{3;t_{n+1}} &= -\mathbf{d}_{3;t_n} + \max\{2\mathbf{d}_{2;t_{n}}, 2\mathbf{d}_{4;t_n}\}, \quad  \mathbf{d}_{j;t_{n+1}}=\mathbf{d}_{j;t_{n}},   &&j\neq 3,n=2k,\\
\mathbf{d}_{4;t_{n+1}} &= -\mathbf{d}_{4;t_n} + \max\{2\mathbf{d}_{2;t_n}, 2\mathbf{d}_{3;t_n}\}, \quad  \mathbf{d}_{j;t_{n+1}}=\mathbf{d}_{j;t_{n}},   &&j\neq 4,n=2k+1.
\end{align*}
And for any $n > 2$, it's clearly that  $\mathbf{d}_{i;t_n}=\mathbf{d}_{i;t_0}\leq \mathbf{0},i\in \{1,2\}$.
We claim that for any $n \in Z^+,k\geq 0$, 
\begin{equation}
\begin{aligned}
\mathbf{d}_{4;t_{n}} &\geq \underset{i \in\{1,2,3,4\}}{\max}\{ \mathbf{d}_{i;t_{n}} \}, &&n=2k,\\
\mathbf{d}_{3;t_{n}} &\geq \underset{i \in\{1,2,3,4\}}{\max}\{ \mathbf{d}_{i;t_{n}} \}, &&n=2k+1,\\
\mathbf{d}_{i;t_{n}} &\geq \mathbf{0}, &&i\in \{3,4\},n\geq 2.
\end{aligned}
\label{eq:5}
\end{equation}
The proof is by induction, analogous to the $\mathbb{G}_{2}^{(*,+)}$ case (Section~\ref{app:G2+}). Hence $y^n=id$ iff $n=0$.

Next, $xyxy^{-1}=g_{21}^{(143),+}$: For any $n \in Z^+,k\geq 0,$ denote 
\[
t_n=
\begin{cases} 
(\mu_{232421})^k(t_0), &n=6k\\   
\mu_{1}(\mu_{232421})^k(t_0), &n=6k+1\\
\mu_{21}(\mu_{232421})^k(t_0), &n=6k+2\\
\mu_{421}(\mu_{232421})^k(t_0), &n=6k+3\\   
\mu_{2421}(\mu_{232421})^k(t_0), &n=6k+4\\
\mu_{32421}(\mu_{232421})^k(t_0), &n=6k+5
\end{cases}
\]
From $\mu_{232421}(B_{t_0})=(B_{t_0})$, the cycle has length 6. For $n>8$,
\begin{align*}
\mathbf{d}_{1;t_{n+1}} &= -\mathbf{d}_{1;t_n} +2\mathbf{d}_{2;t_n}, \quad  \mathbf{d}_{j;t_{n+1}}=\mathbf{d}_{j;t_{n}},   &&j\neq 1,n=6k,\\
\mathbf{d}_{2;t_{n+1}} &= -\mathbf{d}_{2;t_n} + \mathbf{d}_{1;t_n}+\mathbf{d}_{4;t_n}, \quad  \mathbf{d}_{j;t_{n+1}}=\mathbf{d}_{j;t_{n}},   &&j\neq 2,n=6k+1,\\
\mathbf{d}_{4;t_{n+1}} &= -\mathbf{d}_{4;t_n} + 2\mathbf{d}_{2;t_n}, \quad  \mathbf{d}_{j;t_{n+1}}=\mathbf{d}_{j;t_{n}},   &&j\neq 4,n=6k+2,\\
\mathbf{d}_{2;t_{n+1}} &= -\mathbf{d}_{2;t_n} + \mathbf{d}_{4;t_n} +\mathbf{d}_{3;t_n}, \quad  \mathbf{d}_{j;t_{n+1}}=\mathbf{d}_{j;t_{n}},   &&j\neq 2,n=6k+3,\\
\mathbf{d}_{3;t_{n+1}} &= -\mathbf{d}_{3;t_n} + 2\mathbf{d}_{2;t_n}, \quad  \mathbf{d}_{j;t_{n+1}}=\mathbf{d}_{j;t_{n}},   &&j\neq 3,n=6k+4,\\
\mathbf{d}_{2;t_{n+1}} &= -\mathbf{d}_{2;t_n} + \mathbf{d}_{1;t_n}+\mathbf{d}_{3;t_n}, \quad  \mathbf{d}_{j;t_{n+1}}=\mathbf{d}_{j;t_{n}},   &&j\neq 2,n=6k+5.
\end{align*}
Similar to Section~\ref{app:A13}, we obtain
\begin{equation}
\begin{aligned}
\mathbf{d}_{2;t_{n}} &=\mathbf{d}_{2;t_{n-6}} + (2,2,2,2),&&n >8,\\
\mathbf{d}_{2;t_{n}}&\geq \mathbf{0},&&8\geq n>0.
\end{aligned}
\label{eq:6}
\end{equation}
Thus $xyxy^{-1}$ has infinite order.

The proof for $xyxyxy^{-1}x$ is analogous. Consequently all three elements have infinite order.

\subsection*{$(2,2,2,4)$ type, $d_1 \neq d_3$}\label{app:2224_2}
We prove that $bc$ has infinite order. For $n \in Z^+,k\geq 0,$ denote 
\[
t_n=
\begin{cases} 
(\mu_{231213})^k(t_0), &n=6k\\   
\mu_3(\mu_{231213})^k(t_0), &n=6k+1\\
\mu_{13}(\mu_{231213})^k(t_0), &n=6k+2\\
\mu_{213}(\mu_{231213})^k(t_0), &n=6k+3\\   
\mu_{1213}(\mu_{231213})^k(t_0), &n=6k+4\\
\mu_{31213}(\mu_{231213})^k(t_0), &n=6k+5
\end{cases}
\]
$\mu_{231213}(B_{t_0})=(B_{t_0})$, cycle length 6. For $n>6$,
\begin{align*}
\mathbf{d}_{1;t_{n+1}} &= -\mathbf{d}_{1;t_n} +2\mathbf{d}_{2;t_n}, \quad  \mathbf{d}_{j;t_{n+1}}=\mathbf{d}_{j;t_{n}},   &&j\neq 1,n=6k,\\
\mathbf{d}_{2;t_{n+1}} &= -\mathbf{d}_{2;t_n} + \mathbf{d}_{1;t_n}+\mathbf{d}_{4;t_n}, \quad  \mathbf{d}_{j;t_{n+1}}=\mathbf{d}_{j;t_{n}},   &&j\neq 2,n=6k+1,\\
\mathbf{d}_{4;t_{n+1}} &= -\mathbf{d}_{4;t_n} + 2\mathbf{d}_{2;t_n}, \quad  \mathbf{d}_{j;t_{n+1}}=\mathbf{d}_{j;t_{n}},   &&j\neq 4,n=6k+2,\\
\mathbf{d}_{2;t_{n+1}} &= -\mathbf{d}_{2;t_n} + \mathbf{d}_{4;t_n} +\mathbf{d}_{3;t_n}, \quad  \mathbf{d}_{j;t_{n+1}}=\mathbf{d}_{j;t_{n}},   &&j\neq 2,n=6k+3,\\
\mathbf{d}_{3;t_{n+1}} &= -\mathbf{d}_{3;t_n} + 2\mathbf{d}_{2;t_n}, \quad  \mathbf{d}_{j;t_{n+1}}=\mathbf{d}_{j;t_{n}},   &&j\neq 3,n=6k+4,\\
\mathbf{d}_{2;t_{n+1}} &= -\mathbf{d}_{2;t_n} + \mathbf{d}_{1;t_n}+\mathbf{d}_{3;t_n}, \quad  \mathbf{d}_{j;t_{n+1}}=\mathbf{d}_{j;t_{n}},   &&j\neq 2,n=6k+5.
\end{align*}
Similar to Section~\ref{app:A13},
\begin{equation}
\begin{aligned}
\mathbf{d}_{2;t_{n}} &=\mathbf{d}_{2;t_{n-3}} + (1,2,2,0),&&n >3,\\
\mathbf{d}_{2;t_{n}}&\geq \mathbf{0},&&3\geq n>0.
\end{aligned}
\label{eq:7}
\end{equation}
Thus $(bc)^n = id$ iff $n=0$, order infinite.

\subsection*{$(1,1,2,2,2)$ type}\label{app:11222}
We prove that $ab$, $ca$, $cb$ have infinite order.

First, $ab=g_{132}^{id,+}$: For $n>3,k\geq 0,$
\[
t_n=
\begin{cases} 
(\mu_{132})^k(t_0), &n=3k\\   
\mu_2(\mu_{132})^k(t_0), &n=3k+1\\
\mu_{32}(\mu_{132})^k(t_0), &n=3k+2
\end{cases}
\]
$\mu_{132}(B_{t_0})=(B_{t_0})$, cycle length 3. For $n>3$,
\begin{align*}
\mathbf{d}_{2;t_{n+1}} &= -\mathbf{d}_{2;t_n} +\mathbf{d}_{1;t_n}+ \mathbf{d}_{3;t_n}, \quad  \mathbf{d}_{j;t_{n+1}}=\mathbf{d}_{j;t_{n}},   &&j\neq 2,n=3k,\\
\mathbf{d}_{3;t_{n+1}} &= -\mathbf{d}_{3;t_n} + 2\mathbf{d}_{2;t_n}, \quad  \mathbf{d}_{j;t_{n+1}}=\mathbf{d}_{j;t_{n}},   &&j\neq 3,n=3k+1,\\
\mathbf{d}_{1;t_{n+1}} &= -\mathbf{d}_{1;t_n} + 2\mathbf{d}_{2;t_n}, \quad  \mathbf{d}_{j;t_{n+1}}=\mathbf{d}_{j;t_{n}},   &&j\neq 1,n=3k+2.
\end{align*}
Similar to Section~\ref{app:A13},
\begin{equation}
\begin{aligned}
\mathbf{d}_{2;t_{n}} &= (k,2k+1,k,0) ,&&n=3k+1,3k+2,3k+3.
\end{aligned}
\label{eq:8}
\end{equation}
Thus $(ab)^n = id$ iff $n=0$.

Second, $ca=g_{24213421}^{(134),+}$: From $\mu_{1}(B_{t_0})=(B_{t_0})$ with a cycle of length 24. Permuting indices by $(134)$ for $n=8k$ yields a cycle of length 8. For $n \in Z^+,k\geq 0,$
\[
t_n=
\begin{cases} 
(\mu_{24213421})^k(t_0), &n=8k\\   
\mu_{1}(\mu_{24213421})^k(t_0), &n=8k+1\\
\mu_{21}(\mu_{24213421})^k(t_0), &n=8k+2\\
\mu_{421}(\mu_{24213421})^k(t_0), &n=8k+3\\   
\mu_{3421}(\mu_{24213421})^k(t_0), &n=8k+4\\
\mu_{13421}(\mu_{24213421})^k(t_0), &n=8k+5\\
\mu_{213421}(\mu_{24213421})^k(t_0), &n=8k+6\\
\mu_{4213421}(\mu_{24213421})^k(t_0), &n=8k+7
\end{cases}
\]
Similar to Section~\ref{app:A13},
\begin{equation}
\begin{aligned}
\mathbf{d}_{1;t_{n+48}} &= \mathbf{d}_{1;t_{n}} +(9,18,9,0),&& n \geq 0,\\
\mathbf{d}_{1;t_{n}} &\geq   \mathbf{0},&& n\in Z^+.
\end{aligned}
\label{eq:9}
\end{equation}
Hence $(ca)^n = id$ iff $n=0$.

The proof for $cb$ is analogous.

\subsection*{$(2,2,2,2,4)$ type}\label{app:22224}
We prove that $ya$ and $ab$ have infinite order.

First, $ya = g_{2}^{(24),+}$: For $n \in Z^+,k\geq 0,$
\[
t_n=
\begin{cases} 
(\mu_{24})^k(t_0), &n=2k\\   
\mu_2(\mu_{24})^k(t_0), &n=2k+1
\end{cases}
\]
$\mu_{24}(B_{t_0})=(B_{t_0})$, cycle length 2. For $n>2$,
\begin{align*}
\mathbf{d}_{2;t_{n+1}} &= -\mathbf{d}_{2;t_n} + 2\mathbf{d}_{4;t_n}, \quad  \mathbf{d}_{j;t_{n+1}}=\mathbf{d}_{j;t_{n}},   &&j\neq 4,n=2k,\\
\mathbf{d}_{4;t_{n+1}} &= -\mathbf{d}_{4;t_n} + 2\mathbf{d}_{2;t_n}, \quad  \mathbf{d}_{j;t_{n+1}}=\mathbf{d}_{j;t_{n}},   &&j\neq 2,n=2k+1.
\end{align*}
Similar to Sections~\ref{app:G2+} and~\ref{app:A13},
\begin{equation}
\begin{aligned}
\mathbf{d}_{4;t_{n}} &= (0,2k,0,2k-1),&&n=2k,\\
\mathbf{d}_{2;t_{n}} &= (0,2k+1,0,2k) ,&&n=2k+1.
\end{aligned}
\label{eq:10}
\end{equation}
Thus $(ya)^n = id$ iff $n=0$.

Second, $ab = g_{312}^{(13),+}$: For $n \in Z^+,k\geq 0,$
\[
t_n=
\begin{cases} 
(\mu_{132312})^k(t_0), &n=6k\\   
\mu_{2}(\mu_{132312})^k(t_0), &n=6k+1\\
\mu_{12}(\mu_{132312})^k(t_0), &n=6k+2\\
\mu_{312}(\mu_{132312})^k(t_0), &n=6k+3\\   
\mu_{2312}(\mu_{132312})^k(t_0), &n=6k+4\\
\mu_{32312}(\mu_{132312})^k(t_0), &n=6k+5
\end{cases}
\]
$\mu_{132312}(B_{t_0})=(B_{t_0})$, cycle length 6. For $n>6$,
\begin{align*}
\mathbf{d}_{2;t_{n+1}} &= -\mathbf{d}_{2;t_n} +\mathbf{d}_{1;t_n} + \mathbf{d}_{3;t_n}, \quad  \mathbf{d}_{j;t_{n+1}}=\mathbf{d}_{j;t_{n}},   &&j\neq 2,n=6k,\\
\mathbf{d}_{1;t_{n+1}} &= -\mathbf{d}_{1;t_n} + 2\mathbf{d}_{2;t_n}, \quad  \mathbf{d}_{j;t_{n+1}}=\mathbf{d}_{j;t_{n}},   &&j\neq 1,n=6k+1,\\
\mathbf{d}_{3;t_{n+1}} &= -\mathbf{d}_{3;t_n} + 2\mathbf{d}_{2;t_n}, \quad  \mathbf{d}_{j;t_{n+1}}=\mathbf{d}_{j;t_{n}},   &&j\neq 3,n=6k+2,\\
\mathbf{d}_{2;t_{n+1}} &= -\mathbf{d}_{2;t_n} + \mathbf{d}_{1;t_n} +\mathbf{d}_{3;t_n}, \quad  \mathbf{d}_{j;t_{n+1}}=\mathbf{d}_{j;t_{n}},   &&j\neq 2,n=6k+3,\\
\mathbf{d}_{3;t_{n+1}} &= -\mathbf{d}_{3;t_n} + 2\mathbf{d}_{2;t_n}, \quad  \mathbf{d}_{j;t_{n+1}}=\mathbf{d}_{j;t_{n}},   &&j\neq 3,n=6k+4,\\
\mathbf{d}_{1;t_{n+1}} &= -\mathbf{d}_{1;t_n} + 2\mathbf{d}_{2;t_n}, \quad  \mathbf{d}_{j;t_{n+1}}=\mathbf{d}_{j;t_{n}},   &&j\neq 1,n=6k+5.
\end{align*}
Similar to Section~\ref{app:A13},
\begin{equation}
\begin{aligned}
\mathbf{d}_{2;t_{n}} &= (k,2k+1,k,0) ,&&n=3k+1,3k+2,3k+3.
\end{aligned}
\label{eq:11}
\end{equation}
Thus $(ab)^n = id$ iff $n=0$, and both elements have infinite order.

\FloatBarrier

\end{document}